%% file: Infinite_time_multiradial_arXiv.tex
\documentclass[a4paper, 11pt]{article}
\usepackage{amssymb,amsmath,amsthm,lmodern,cancel}
\usepackage{amsfonts}
\usepackage{accents}

\usepackage[activate={true,nocompatibility},final,tracking=true,kerning=true,spacing=true,factor=1100,stretch=15,shrink=15]{microtype}
\microtypecontext{spacing=nonfrench}

\usepackage{graphicx}
\usepackage[english]{babel}
\usepackage[utf8]{inputenc}
\usepackage[T1]{fontenc}
\usepackage{textcomp}
\usepackage[dvipsnames]{xcolor}
\usepackage{soul}
\usepackage[shortlabels]{enumitem}
\usepackage{thmtools, mathtools,verbatim} \usepackage{upgreek}
\usepackage{thm-restate}
\usepackage{titlesec} 
\usepackage[normalem]{ulem} 
\usepackage{mathrsfs} 

\usepackage[colorlinks=false,hyperindex=true,hypertexnames=false]{hyperref}
\usepackage{cleveref}

\usepackage[font=normal]{caption}

\usepackage{thmtools}

\usepackage[bottom]{footmisc}

\title{Multiradial Schramm-Loewner evolution: \\ Infinite-time large deviations and transience}  
\date{}

\author{
Osama Abuzaid\thanks{Department of Mathematics and Systems Analysis, Aalto University, Finland.  \protect\url{osama.abuzaid@aalto.fi}} , \, 
Vivian Olsiewski Healey\thanks{Department of Mathematics, Texas State University, US. \protect\url{healey@txstate.edu} } , \,
and \,
Eveliina Peltola\thanks{Department of Mathematics and Systems Analysis, Aalto University, Finland; and \\ Division of Mathematics, University of Cologne, Germany.   \protect\url{eveliina.peltola@aalto.fi}}
}

\setlist[enumerate]{topsep = 1ex, leftmargin=1cm, itemsep= -2pt}

\let\OLDthebibliography\thebibliography
\renewcommand\thebibliography[1]{
\OLDthebibliography{#1}
\setlength{\parskip}{1pt}
\setlength{\itemsep}{2pt}
}

\allowdisplaybreaks

\newtheorem{thm}{Theorem}[section]

\newtheorem{lem}[thm]{Lemma}
\newtheorem{prop}[thm]{Proposition}

\newtheorem{theorem}{Theorem}

\newtheorem{lemA}[theorem]{Lemma}

\newtheorem{proposition}[theorem]{Proposition}

\theoremstyle{definition} 
\newtheorem{df}[thm]{Definition}

\newtheorem{remark}[thm]{Remark}

\numberwithin{equation}{section}
\numberwithin{figure}{section}

\input{macro-arxiv}

\begin{document}
\maketitle

\begin{abstract}
In previous work~\cite{AHP:Large_deviations_of_DBM_and_multiradial_SLE}, we proved a finite-time large deviation principle in the Hausdorff metric for multiradial Schramm-Loewner evolution, $\SLE_\kappa$, as $\kappa \to 0$,
with good rate function being the multiradial Loewner energy. 
Here, we extend this result to infinite time in the topology of common-capacity-parameterized curves, and streamline the proof. 
A~main step is to derive detailed escape probability estimates for multiradial $\SLE_\kappa$ curves in the common parameterization, 
which extend the single-curve estimates achieved in~\cite{Abuzaid-Peltola:Large_deviations_of_radial_SLE0}.
As a by-product, we also get that multiradial $\SLE_\kappa$ curves, with $\kappa \leq 8/3$, are transient at their common terminal point, generalizing~\cite{Field-Lawler:Escape_probability_and_transience_for_SLE, Healey-Lawler:N_sided_radial_SLE}.
As a corollary to the LDP result, we obtain explicit asymptotics of the Brownian loop measure interaction term for finite-energy radial multichords, 
which is linear in the capacity-time and coincides with a certain choice of a cocycle for the Virasoro algebra.
\end{abstract}

\newpage

\tableofcontents

\newpage

\section{Introduction}

This article is a sequel to~\cite{AHP:Large_deviations_of_DBM_and_multiradial_SLE}, 
where we proved a finite-time large deviation principle (LDP) in the Hausdorff metric for multiradial Schramm-Loewner evolution, $\SLE_\kappa$, as $\kappa \to 0+$.
In the present work, we shall improve the result in our earlier work in several ways.
First, we prove a finite-time LDP for multiradial $\SLE_{0+}$ in the topology of common-capacity-parameterized curves  (\Cref{thm:n-radial-LDP-finite}).  
This establishes the LDP in a stronger topology and enables us to extend the result to infinite time,
which is the main result of the present article (\Cref{thm:n-radial-LDP-infinite-time}).
The associated rate function is the \emph{multiradial Loewner energy}, which can be written either in terms of 
the interacting driving processes of the multiple random curves (in the spirit of the Freidlin-Wentzell theorem),
or in terms of a Brownian loop measure (in the spirit of conformal restriction covariance of $\SLE_\kappa$ curves).

Let $\Pr^\kappa$ be the law of the $n$-radial $\SLE_\kappa$ curves $\bgamma = (\gamma^1, \gamma^2, \ldots, \gamma^n)$. 
Informally, the LDP statement means that, with some choice of topology for the curves, one has
\begin{align} \label{eq:asy}
\Pr^\kappa \big[ \bgamma \textnormal{ lie in a neighborhood of given curves } \boldsymbol{\eta} \big] \quad \overset{\kappa \to 0\!+}{\sim} \quad 
\exp\bigg(\! -\frac{\nBessel(\boldsymbol{\eta})}{\kappa}\bigg)  ,
\end{align}
where $\nBessel$ is a non-negative functional on the space of curves, called the \emph{rate function}. 
It describes the exponentially fast concentration of the $n$-radial $\SLE_\kappa$ probability measure at its minima:
note that the righthand side of~\eqref{eq:asy} typically tends to zero as $\kappa \to 0$, unless $\nBessel(\boldsymbol{\eta}) = 0$.
Thus, in addition to verifying the asymptotics in~\eqref{eq:asy} rigorously in various topologies, 
it is of considerable interest to know properties of the rate function, or at least whether it admits a minimum.
In fact, rate functions for SLE variants have a deep geometric meaning (which is beyond the scope of the present article); see~\cite{Wang:SLE_LDP_survey}.
Let us note, however, that the rate function for $n$-radial $\SLE_\kappa$ displays a new phenomenon not yet observed for other SLE variants:
it inherently carries a \emph{renormalization}, due to the fact that the $n$-radial $\SLE_\kappa$ is not absolutely continuous with respect to the product measure of $n$ independent radial $\SLE_\kappa$ curves.
We shortly make this renormalization explicit, and relate it to the conformal restriction covariance of $\SLE_\kappa$ curves,
or in other words, to the conformal anomaly in the associated conformal field theory (or in representation theoretic terms, a specific choice of the Virasoro cocycle, which may be of independent interest).

\subsection{Setup and the LDP}

We consider radial multichords in the disk $\bD$, i.e.,
$n$-tuples of simple curves starting at distinct boundary points, targeted at the origin, which only intersect each other at the origin (\Cref{def:radial_multichord}).
We let $\commonpaths[\infty]$ denote the space of all simple $n$-radial multichords, 
parameterized by common capacity time (see \Cref{subsec:radial_multichords}), endowed with the metric
\begin{align}\label{def:toplogy-with-param-full}
\dcommonpaths[](\bgamma_1, \bgamma_2) 
:= \sup_{t \geq 0} \sum_{j=1}^n \big| \gamma_1^j (t) - \gamma_2^j(t) \big| 
, \qquad \bgamma_1, \bgamma_2  \in \commonpaths[\infty] .
\end{align} 

Two of us proved in~\cite[Theorem~1.2]{Abuzaid-Peltola:Large_deviations_of_radial_SLE0} 
that the (single) radial $\SLE_\kappa$ process satisfies a strong LDP in the capacity parameterization, where the rate function is obtained from the Dirichlet energy of the Loewner driving function, and conventionally dubbed ``radial Loewner energy.'' 
We refer to the introduction of the article~\cite{Abuzaid-Peltola:Large_deviations_of_radial_SLE0} for a thorough discussion of the scope of this result, and for references to related works.
It immediately follows that $n$ independent radial $\SLE_\kappa$ curves $\tilde \bgamma = (\tilde \gamma^1, \ldots, \tilde \gamma^n)$ satisfy a similar LDP with rate function
\begin{align*}
\nDindepenergy(\tilde \bgamma)  := \lim_{\tilde T \to \infty} \nDindepenergy_{\tilde T}(\tilde \bgamma|_{[0,\tilde T]}) ,
\end{align*}
defined in terms of the Loewner driving functions $(\tilde\vartheta^1, \ldots, \tilde\vartheta^n)$ of the independent curves:
\begin{align}\label{eqn:nDirichlet_energy_def}
\nDindepenergy_{\tilde T}(\tilde \bgamma) 
:= \; &
\begin{dcases} 
\sum_{j=1}^n \tfrac{1}{2} \int_0^{\tilde T} \big| \tfrac{\ud}{\ud t} \tilde \vartheta_t^j \big|^2 \ud t ,
\quad & \textnormal{if } \tilde \vartheta^j  
\textnormal{ each is absolutely continuous on $[0,\tilde T]$,} \\
\infty , & \textnormal{otherwise,}
\end{dcases}
\end{align}
where $\tilde\vartheta^j$ is the Loewner driving function of the curve $\tilde\gamma^j$, for each $j \in \{1,\ldots,n\}$.

The law $\Pr^\kappa$ of the $n$-radial $\SLE_\kappa$ process in the common parameterization (see \Cref{def:n-radial_SLE_driving_functions}) is a probability measure on the space $(\commonpaths[\infty],\dcommonpaths[])$. 
Let us state our LDP result.

\begin{restatable}{thm}{LDPinfinite} \label{thm:n-radial-LDP-infinite-time}
The family $(\Pr^\kappa)_{\kappa > 0}$ of laws of the $n$-radial $\SLE_\kappa$ curves 
satisfy an LDP on $\commonpaths[\infty]$ 
with good rate function $\nBessel$\textnormal{:}
for every open $G \subset \commonpaths[\infty]$ and closed $F \subset \commonpaths[\infty]$, we have
\begin{align}\label{eq:LDP_bounds}
\begin{split}
\liminf_{\kappa \to 0+}\kappa\log\Pr^\kappa[G] \ge\;& 
-\inf_{\bgamma \in G}\nBessel(\bgamma) , \\
\limsup_{\kappa \to 0+}\kappa\log\Pr^\kappa[F] \le\;& 
-\inf_{\bgamma \in F}\nBessel(\bgamma), 
\end{split}
\end{align}
where the rate function $\nBessel \colon \commonpaths[\infty] \to [0,+\infty]$ can be written 
in terms of  the Dyson-Dirichlet energy of the multiradial Loewner driving function $\btheta$~of~$\bgamma$,
\begin{align}\label{eq:Dyson-Dirichlet_energy_infty}
\hspace*{-4mm}
\nBessel(\bgamma) 
:= \; &
\begin{dcases} 
\tfrac{1}{2} \int_0^\infty \sum_{j=1}^n \bigg| \tfrac{\ud}{\ud s} \theta^j_s - 2 \sum_{\substack{1 \leq i \leq n \\[.1em] i\neq j}} \cot \Big( \tfrac{\theta_s^j - \theta_s^i}{2} \Big) \bigg|^2 \ud s ,
\quad & \textnormal{if } \btheta 
\textnormal{ is abs. cont. on $[0,T]$,} \\
\infty , & \textnormal{otherwise.}
\end{dcases}
\end{align}
\end{restatable}

To derive the strong LDP result, we use a convenient method for tilted measures: we apply a generalized version of Varadhan’s lemma 
(which we present for completeness in \Cref{thm:Varadhan-general} in Appendix~\ref{app:Varadhan}),
allowing us to transfer the aforementioned LDP for independent curves to the interacting curves.
However, there are two pitfalls we need to be careful with: differing time-parametrizations, and the restriction of the measures to non-colliding curves.
We consider the time-parametrizations in \Cref{sec:common-and-independent-parameterizations} (see in particular \Cref{thm:time-change-properties-better})
and derive the LDP results in \Cref{sec:finite-time-LDP} 
(establishing the finite-time case in \Cref{thm:n-radial-LDP-finite}, which already requires the restriction of the measures to non-colliding curves via \Cref{thm:restricted-LDP})
and in \Cref{sec:LDP} (dealing with the infinite-time limit $T \to \infty$).
The key ingredients to prove the infinite-time LDP are precise escape probility estimates (established in \Cref{thm:conditional-escape-probability} and \Cref{thm:escape-estimates-small-kappa}).
In the same spirit, we can also show that $n$-radial $\SLE_\kappa$ curves 
are transient at their common terminal point, generalizing~\cite{Field-Lawler:Escape_probability_and_transience_for_SLE, Healey-Lawler:N_sided_radial_SLE}:

\begin{restatable}{thm}{Transience}\label{thm:transience}
Fix $\kappa \in (0,8/3]$. 
$n$-radial $\SLE_\kappa$ is transient: $\underset{t \to \infty} {\lim} \, \bgamma(t) = \boldsymbol{0}$ almost surely.
\end{restatable}

A single radial $\SLE_\kappa$ curve $\gamma$ is known to be transient for all $\kappa\in (0,8)$~\cite{Lawler:Continuity_of_radial_and_two-sided_radial_SLE_at_the_terminal_point, Field-Lawler:Escape_probability_and_transience_for_SLE}.
For multiple curves, 
the recent result~\cite[Theorem~1.3]{HPW:Multiradial_SLE_with_spiral} considers multiradial $\SLE_\kappa$ (with spiral) in terms of the conditional and marginal laws of its sub-collections of curves, and establishes transience of the process.
The proof in~\cite[Theorem~1.3]{HPW:Multiradial_SLE_with_spiral} proceeds via heavy analysis of $\SLE_\kappa$ partition functions, and the argument implying transience is a consequence of the transience of the half-watermelon $\SLE_\kappa$ process~\cite[Section~4]{Miller-Sheffield:Imaginary_geometry1}, 
which in turn relies on the flow-line (``imaginary geometry'') coupling of $\SLE_\kappa$ with the Gaussian free field~\cite{Miller-Sheffield:Imaginary_geometry1}.
Our arguments leading to \Cref{thm:transience} are direct and overwhelmingly simpler, 
but we assume $\kappa \leq 8/3$,
so that the central charge is negative and we can ignore the Brownian loop measure term (which diverges as $T \to \infty$, causing difficulties to bound the Radon-Nikodym derivative in Equation~\eqref{eqn:RNloop}).
We have not attempted to improve the range of $\kappa$ in \Cref{thm:transience}, which might be possible (and interesting) to do.

\subsection{The multiradial Loewner energy}

For a radial multichord $\bgamma$, the multiradial Loewner energy can be defined in various ways.
On the one hand, it has an explicit expression in terms of the interacting driving processes of the multiple random curves, Equation~\eqref{eq:Dyson-Dirichlet_energy_infty}, 
which relies on adding the interaction into the Dirichlet energy as a drift,
and was a key object in~\cite{Healey-Lawler:N_sided_radial_SLE, AHP:Large_deviations_of_DBM_and_multiradial_SLE}.
This is the \emph{dynamical} point of view of constructing multiple radial curves via Loewner theory. 
On the other hand, it can be written in terms of a Brownian loop measure $\BLoop_{\bD}$ (see Equation~\eqref{eq:BLM_energy}), 
encoding the interaction of the curves in a more \emph{global} manner~\cite{Lawler:Partition_functions_loop_measure_and_versions_of_SLE,Healey-Lawler:N_sided_radial_SLE}:
\begin{align*}
\cL(\bgamma[0,T]) :=
\sum_{j=2}^n \BLoop_{\bD} \big[ \ell \colon \gamma^j[0,T] \cap \ell \ne \emptyset \textnormal{ for at least two distinct } j \big] .
\end{align*}
The energy functional also involves the semiclassical partition function, encoding the conformal moduli. 
It is a positive function $\LogPartF \colon \chamber \to (0,\infty)$, 
\begin{align} \label{eq:LogPartF}
\LogPartF(\boldsymbol{\alpha}) 
:= - 2 \underset{1\leq i<j\leq n}{\sum} \log \sin\bigg(\frac{\alpha^j-\alpha^i}{2}\bigg) ,
\end{align}
defined on the subset $\chamber \subset (\bR/2\pi\bZ)^n$ of the torus with periodic boundary conditions comprising the elements admitting representatives 
$\boldsymbol{\alpha} = (\alpha^1, \ldots, \alpha^n)$ which satisfy
\begin{align} \label{eq: torus ordering}
\alpha^1<\alpha^2< \cdots< \alpha^n <\alpha^1+2\pi.
\end{align}
(Throughout, we use the convention that $\theta^{n+j} = \theta^j+2\pi$ for all $j$.)

\begin{restatable}{thm}{BLMForm}\label{thm:BLMForm}
Fix a starting point $\btheta_0 \in \chamber$. 
The rate function $\nBessel \colon \commonpaths[\infty] \to [0,+\infty]$ can be written in terms of the Brownian loop measure and the semiclassical partition function as
\begin{align}\label{eq:BLM_energy}
\nBessel(\bgamma) \; = \; \nDindepenergy(\bgamma) \; - \; \LogPartF(\btheta_0) \; + \;  \inf_{\boldsymbol{\alpha} \in \chamber}\LogPartF(\boldsymbol{\alpha}) \; + \;  \limsup_{T \to \infty} 12 \cLrenorm(\bgamma[0,T]) ,
\end{align}
where
\begin{align}\label{eq:LogPartFMin}
\inf_{\boldsymbol{\alpha} \in \chamber}\LogPartF(\boldsymbol{\alpha}) 
= - 2 \underset{1\leq i<j\leq n}{\sum} \log \sin\big(\tfrac{\pi(j-i)}{n}\big) ,
\end{align}
and 
\begin{align}\label{eq:cLrenorm}
\cLrenorm(\bgamma[0,T]) \, := \, \cL(\bgamma[0,T])  \, - \,  \tfrac{1}{24} (n + 4) (n - 1) \, n \, T.
\end{align}
\end{restatable}

For the driving processes, we already know by our earlier result~\cite[Theorem~1.18]{AHP:Large_deviations_of_DBM_and_multiradial_SLE}
that finite-energy drivers converge in the long term to an equally spaced configuration. 
The convergence rate is exponential for zero-energy drivers, but can be arbitrarily slow for general finite-energy drivers; 
and the limit is static for zero-energy drivers, but can exhibit slow spiraling for general finite-energy drivers (see~\cite[Examples~4.8~\&~4.10]{AHP:Large_deviations_of_DBM_and_multiradial_SLE}).
These configurations minimize the semiclassical partition function. 
We record this fact here for two reasons: it is of independent interest, and we shall use it in the proof of \Cref{thm:BLMForm}.

\begin{restatable}{prop}{minimum}\label{thm:minimum}
For any finite-energy $n$-radial multichord, i.e., if $\nBessel(\bgamma) < \infty$, we have
\begin{align*}
\lim_{T \to \infty}\LogPartF(\btheta_T) = \inf_{\boldsymbol{\alpha} \in \chamber}\LogPartF(\boldsymbol{\alpha}) 
= - 2 \underset{1\leq i<j\leq n}{\sum} \log \sin\big(\tfrac{\pi(j-i)}{n}\big) ,
\end{align*}
where $\btheta$ is the multiradial Loewner driving function~of~$\bgamma$. 
\end{restatable}

Whether or not the limit as $T \to \infty$ of the righthand side of Equation~\eqref{eq:BLM_energy} exists is not clear a priori,
because the term involving the Brownian loop measure diverges due to the presence of small loops at the vicinity of the common endpoint of the curves. 
As a byproduct of our result, we obtain the explicit asymptotics of this divergence. 
Interestingly, it  coincides with the Gel'fand-Fuks cocycle~\cite{Gelfand-Fuchs:Cohomologies_of_the_Lie_algebra_of_vector_fields_on_the_circle} 
of the Virasoro algebra with a certain (atypical) normalization (more typical ones are multiples of $n^3$ and $(n^2-1)n = (n + 1) (n - 1) n$).

\begin{restatable}{cor}{BLM}\label{thm:BLM}
For any finite-energy $n$-radial multichord, i.e., if $\nBessel(\bgamma) < \infty$, we have
\begin{align*}
\lim_{T \to \infty} \frac{\cL(\bgamma[0,T])}{T} = \frac{(n + 4) (n - 1) n}{24} .
\end{align*}
\end{restatable}

We prove \Cref{thm:BLMForm}, \Cref{thm:minimum}, and \Cref{thm:BLM} in \Cref{subsec:multiradial_Loewner_energy}.

\subsection*{Acknowledgments}

Part of this project was performed while the authors were participating in a program hosted by the Hausdorff Research Institute for Mathematics (HIM) in Bonn, Germany, in Summer 2025,  
supported by the Deutsche Forschungsgemeinschaft (DFG, German Research Foundation) under Germany's Excellence Strategy EXC-2047/1-390685813.
O.A.~is supported by the Academy of Finland grant number 340461 ``Conformal invariance in planar random geometry.''
V.O.H.~is partially supported by an AMS Simons Research Enhancement Grant for PUI Faculty.
This material is part of a project that has received funding from the  European Research Council (ERC) under the European Union's Horizon 2020 research and innovation programme (101042460): 
ERC Starting grant ``Interplay of structures in conformal and universal random geometry'' (ISCoURaGe) 
and from the Academy of Finland grant number 340461 ``Conformal invariance in planar random geometry.''
E.P.~is also supported by 
the Academy of Finland Centre of Excellence Programme grant number 346315 ``Finnish centre of excellence in Randomness and STructures (FiRST)'' 
and by the Deutsche Forschungsgemeinschaft (DFG, German Research Foundation) project number 390534769 ``Matter and Light for Quantum Computing (ML4Q).

\section{Comparison of common and independent parametrizations}
\label{sec:common-and-independent-parameterizations}

In \Cref{sec:finite-time-LDP}, we will transfer a large deviations result of~\cite{Abuzaid-Peltola:Large_deviations_of_radial_SLE0} for independent $\SLE_{0+}$ curves to the $n$-radial case.  
However, the definition of $n$-radial $\SLE_\kappa$ (\Cref{def:n-radial_SLE_driving_functions}) 
uses the \emph{common parameterization} for simple radial multichords.
Accordingly, we will need explicit bounds governing the time-change from independent to common parameterization and vice versa (see \Cref{thm:time-change-properties-better}). 
This section is dedicated to such preliminary results.

\subsection{Hulls and mapping-out functions}

A \emph{hull} is a compact set $K \subset \overline{\bD}$ such that $\bD\smallsetminus K$ is simply connected, $0 \in \bD\smallsetminus K$, and $\overline{K \cap \bD} = K$. 
The \emph{mapping-out function} of the hull $K$ is defined as the unique conformal bijection  
$g_K \colon \bD \smallsetminus K \to \bD$ satisfying the normalization $g_K(0) = 0$ and $g_K'(0) > 0$. 
By Schwarz lemma, we have $g_K'(0) \geq 1$.
The (logarithmic) capacity of $K$ is defined as $\rcap(K) := \log g'_K(0) \geq 0$.
By Schwarz reflection, the map $g_K$ extends to a conformal map on $D_K := \hat{\bC}\smallsetminus (K \cup K^*)$, 
whose image is a slit sphere $S_K :=  g_K(D_K)$, where $K^*$ is the reflection of $K$ across the unit circle and $\hat{\bC} = \bC \cup \{\infty\}$ is the Riemann sphere.  
Throughout, we also write $e^{\ii h_K(z)} := g_K(e^{\ii z})$, where the covering map is chosen such that $h_K(0) \in [0,2\pi)$. 

For a domain $D \subsetneq \bC$, we let $\hmeas{D}(I; \cdot)$ denote the harmonic measure of $I \subset \partial D$. 

\begin{lem}\label{thm:h-derivative-bound-better}
Fix a hull $K \subset \overline{\bD}$.  For any $x \in \bR$ such that $\dist(e^{\ii x}, K) > 0$, we have
\begin{align*}
\tfrac{1}{4} \, \sin \Big(\pi\min\Big\{\hmeas{\bD\smallsetminus K}(I^+_x; 0),\hmeas{\bD\smallsetminus K}(I^-_x; 0)\Big\}\Big)
\, \le \, | h_K'(x) |  \, \le \, e^{-\rcap(K)}  \, = \, \tfrac{1}{g_K'(0)} ,
\end{align*}
where $I^\pm_x$ are the two components of $\partial\bD \smallsetminus (K\cup\{e^{\ii x}\})$ adjacent to $e^{\ii x}$,
and $D_K := \hat{\bC}\smallsetminus (K \cup K^*)$.
\end{lem}

It is an easy consequence of the Julia-Wolff lemma (see~\cite[Proposition~4.13]{Pommerenke:Boundary_behaviour_of_conformal_maps})
that $|h_K'(x)| = | g_K'(e^{\ii x}) | \leq 1 \leq g_K'(0) = e^{\rcap(K)}$, so the map $g_K$ is contracting at the boundary and expanding at the origin.
For our results, however, we need a stronger upper bound. 
 
\begin{proof}
By Koebe $1/4$-theorem (e.g.~\cite[Theorem~2.10.6]{AIM:Elliptic_partial_differential_equations_and_quasiconformal_mappings_in_the_plane}), 
\begin{align*}
\frac{\dist(g_K(z), \partial S_K )}{4\,\dist(z, \partial D_K )}
\leq | g_K'(z) | 
, \qquad z \in D_K .
\end{align*}
Noting that for $z = e^{\ii x}$, we have $\dist(z, \partial D_K ) \leq 2$ and
\begin{align*}
\dist(g_K(e^{\ii x}), \partial S_K ) 
= 2 \, \sin \Big(\pi\min\Big\{\hmeas{\bD\smallsetminus K}(I^+_x; 0),\hmeas{\bD\smallsetminus K}(I^-_x; 0)\Big\}\Big) ,
\end{align*}
which gives the desired lower bound for $| g_K'(e^{\ii x}) | = |h_K'(x)|$.
Consider then the inverse map $f_K := g_K^{-1} \colon \bD \to \bD \smallsetminus K$.
Recall that single-slit mappings are dense in the space of all univalent functions on $\bD$~\cite[Theorem~3.2 and its proof]{Duren:Univalent_functions}, in the topology of uniform convergence on compact subsets of $\bD$. 
Hence, there exists a sequence $(f_{K_k})_{k \in \bN}$ of slit mappings approximating $f_K$ uniformly on compact subsets of $\bD$, where $K_k = \eta[0,T_k]$ are slits, $\eta \colon (0,T_k] \to \bD$ simple curves parameterized by capacity, 
so that $g_{\eta[0,t]}'(0) = e^t$, and 
\begin{align*}
\partial_t h_{\eta[0,t]}(x) = \; & \cot\bigg(\frac{h_{\eta[0,t]}(x)-\theta_t}{2}\bigg) , \qquad t \in [0,T_k] ,
\end{align*}
where $(\theta_t)_{t \in [0,T_k]}$ is their Loewner driving function. It thus suffices to prove that 
\begin{align}\label{eq:derivative_desired}
|h_{K_k}'(x)| = | g_{K_k}'(e^{\ii x}) | \, \le \, e^{-\rcap(K_k)} \, = \, e^{-T_k} .
\end{align}
Integrating the differential inequality for the spatial derivative of $h_t(x)$,
\begin{align*}
\partial_t h_t'(x) = \; & - h_t'(x) \, \csc^2\bigg(\frac{h_t(x)-\theta_t}{2}\bigg) \leq - h_t'(x) ,
\end{align*}
with initial condition $h_0'(x)=1$ yields $h_t'(x) \leq e^{-t}$.
Taking $t \uparrow T_k$, we obtain~\eqref{eq:derivative_desired}.
\end{proof}

\subsection{Radial multichords and Loewner evolution}
\label{subsec:radial_multichords}

\begin{df} \label{def:radial_multichord}
Fix distinct points $x^1, \ldots, x^n \in \partial \bD$. 
We call an $n$-tuple $\bgamma = (\gamma^1, \ldots, \gamma^n)$ such that $\gamma^1, \ldots, \gamma^n$ are curves\footnote{Here, we think of $\gamma^j$ as a continuous map from $(0,\infty)$ to $\overline{\bD}$, started at $\gamma^j(0)$ and ending at $\lim_{t\to \infty} \gamma^j(t)$. Note that the definition of a radial multichord allows the curves to intersect.} in $\overline{\bD}$ with  
$\gamma^j(0) = x^j$, and $\lim_{t\to \infty} \gamma^j(t)=0$ for each $j$, 
a~\emph{radial multichord} in $(\bD;x^1, \ldots, x^n)$. 
We naturally identify $\bgamma$ with the union $\cup_j \gamma^j \subset \overline{\bD}$. 
\end{df}

Given an $n$-tuple of radial curves on a time-window $[0,T]$, we obtain an $n$-radial multichord in the strict sense of \Cref{def:radial_multichord} by attaching to the tip $\gamma^j(T)$ of each curve, say, 
a hyperbolic line segment of $\bD\smallsetminus\big(\bigcup_{i=1}^n \gamma^i[0, T]\big)$ from $\gamma^j(T)$ to $0$ and parameterizing $\gamma^j$ such that $\lim_{t\to \infty} \gamma^j(t)=0$ for each $j$. 
We call the radial multichord $\bgamma[0,T]$ \emph{simple} if its each component $\gamma^j$ is injective, 
$\bgamma(0,T] \subset \bD$, and furthermore $\gamma^j[0,T] \cap \gamma^k[0,T]$ for all $j \neq k$.
We call the full $n$-radial multichord $\bgamma[0,\infty)$ \emph{simple} if $\bgamma[0,T]$ is simple for all $T \in [0,\infty)$.
Hence, simple radial multichords only intersect at their common endpoint at the origin.

Simple radial multichords can be generated using Loewner evolution. 
The (multiradial) \emph{Loewner chain} $(g_t)_{t \geq 0}$ with driving functions (``drivers'') $\exp(\ii \theta^1_t), \ldots, \exp(\ii \theta^n_t) \in \partial \bD$
and the \emph{common parameterization} is the family of solutions to the boundary value problem
\begin{align}\label{eqn:multiradial_Loewner_1common}
\partial_t g_t(z) = g_t(z) \sum_{j=1}^n \frac{e^{\ii \theta^j_t} + g_t(z)}{e^{\ii \theta^j_t} - g_t(z)},
\qquad g_0(z) = z ,
\qquad z \in \overline{\bD} , \quad t \geq 0 .
\end{align}
We identify the drivers with continuous functions $t \mapsto \btheta_t = (\theta^1_t, \ldots, \theta^n_t)$ from $[0,\infty)$ to $\chamber$, that is, the subset $\chamber \subset (\bR/2\pi\bZ)^n$ of the torus with periodic boundary conditions comprising the elements admitting representatives 
$\boldsymbol{\alpha} = (\alpha^1, \ldots, \alpha^n)$ which satisfy~\eqref{eq: torus ordering}.

We now fix a starting point $\btheta_0 \in \chamber$ throughout, and write $e^{\ii \btheta_0} = (e^{\ii \theta_0^1}, \ldots, e^{\ii \theta_0^n}) \in (\partial \bD)^n$. 
We are interested in the case where $(g_t)_{t \geq 0}$ is generated by a simple radial multichord $\bgamma \colon (0,T] \to \bD^n$ in $(\bD;e^{\ii \btheta_0})$, 
so that $g_t \colon \bD \smallsetminus K_t \to \bD$ are mapping-out functions of hulls $K_t = \bgamma[0,t]$ for each $t$. 
The parameterization in~\eqref{eqn:multiradial_Loewner_1common} guarantees that $g_t'(0) = e^{nt}$,
so that each component is locally growing at the same rate (see~\cite{Healey-Lawler:N_sided_radial_SLE} for more details),
and the total capacity of $\bgamma[0,t]$ equals $\rcap(\bgamma[0,t]) = \log g'_t(0) = nt$.
We let $\commonpaths[T] \ni \bgamma$ denote the space of such simple $n$-radial multichords, 
endowed with the metrizable topology obtained by parameterizing the curves $\bgamma$ by their common capacity-time and using the metric
\begin{align}\label{def:toplogy-with-param}
\dcommonpaths[T](\bgamma_1, \bgamma_2) 
:= \sup_{t \in [0,T]} \; \sum_{j=1}^n \big| \gamma_1^j (t) - \gamma_2^j(t) \big| 
, \qquad \bgamma_1, \bgamma_2 \in \commonpaths[T] .
\end{align}

For $\bgamma \in \commonpaths[T]$, we define the time-change function\footnote{If clear from context, we leave the dependence of $\bindeptime = \bindeptime_\bgamma$ on $\bgamma$ implicit.}  
$\bindeptime = \bindeptime_\bgamma = (\indeptime_{\bgamma}^j)_{j=1}^n \colon [0, T] \to [0, \infty)^n$,
\begin{align} \label{eqn:def-of-time-change-sigma}
\indeptime_\bgamma^j(t) := \rcap(\gamma^j[0,t]) , \qquad 1 \le j \le n , \; t \in [0,T] .
\end{align}
Thus, $t \mapsto \indeptime_{\bgamma}^j(t)$ corresponds to the capacity parameterization of the individual curve $\gamma^j$ in $\bD$, independently of the other curves. 
In contrast, for $\bindepT = (\indepT^1,\ldots,\indepT^n) \in [0,\infty)^n$, we let 
$$\Dindpaths[\bindepT] \ni \tilde\bgamma = (\tilde\gamma^1,\ldots,\tilde\gamma^n)$$ 
denote the product space of $n$ simple radial curves $\tilde\gamma^j \colon (0,\indepT^j] \to \bD$ such that $\tilde\gamma^j(0) = \exp(\ii \theta_0^j)$ with independent parameterizations by capacity: $s = \rcap(\tilde\gamma^j[0,s])$, for $s \in [0,\indepT^j]$.
If the curves in $\tilde\bgamma$ are distinct, we see that there exists $T \in [0,\infty)$ such that $\tilde\bgamma|_{\bindeptime_{\bgamma}[0,T]} \circ \bindeptime_{\bgamma} = \bgamma$ for some $\bgamma \in \commonpaths[T]$.
We endow the space $\Dindpaths[\bindepT]$ of independent curves with the product topology induced by the capacity-parameterized metrics 
\begin{align}\label{def:toplogy-inped-param}
\dindpaths[\bindepT](\tilde{\bgamma}_1, \tilde{\bgamma}_2) 
:= \sup_{(s^1,\ldots,s^n) \in [\boldsymbol{0},\bindepT]} \; \sum_{j=1}^n \big| \tilde{\gamma}_1^j (s^j) - \tilde{\gamma}_2^j(s^j) \big| 
, \qquad \tilde{\bgamma}_1, \tilde{\bgamma}_2  \in \Dindpaths[\bindepT] .
\end{align}

\subsection{Multiradial Schramm-Loewner evolution}

We now recall the setup from~\cite{Healey-Lawler:N_sided_radial_SLE} for $n$-radial $\SLE_\kappa$ processes. 

\begin{df} \label{def:n-radial_SLE_driving_functions}
Fix $\btheta_0 \in \chamber$. 
For each parameter $0<\kappa\leq 4$, 
\emph{$n$-radial $\SLE_\kappa$ with the common parameterization} started from $e^{\ii \btheta_0}\in (\partial \bD)^n$ 
is the random radial multichord $\bgamma$ for which the mapping-out functions $g_t \colon \bD \smallsetminus \bgamma[0,t] \to \bD$ satisfy  Equation~\eqref{eqn:multiradial_Loewner_1common}
with drivers obtained from the unique strong solution $\btheta_t := (\theta^1_t, \ldots, \theta^n_t)$ to the SDEs
\begin{align}\label{eq:SDEs}
\ud \theta^j_t = \; & 2 \, \sum_{\substack{1 \leq i \leq n \\[.1em] i\neq j}} \cot \bigg( \frac{\theta^j_t - \theta^i_t}{2} \bigg) \ud t 
+ \sqrt \kappa \ud W^j_t , \qquad \textnormal{for all } j \in \{1,\ldots,n\}, 
\end{align}
where $W^1_t, \ldots, W^n_t$ are independent standard Brownian motions. 
\end{df}

The above system~\eqref{eq:SDEs} of SDEs is, a priori, only defined up to the collision time
\begin{align*}
\blowuptime := \; & \inf \Big\{t \geq 0 \; \colon \; \min_{1 \leq i < j \leq n} \, \big| e^{\ii \theta^i_t} - e^{\ii \theta^j_t} \big| = 0 \Big\} ,
\end{align*}
which for the $n$-radial $\SLE_\kappa$ process is almost surely infinite 
(see~\cite[Corollary~2.5]{AHP:Large_deviations_of_DBM_and_multiradial_SLE}).

\begin{lemA}[{\cite[Lemma~3.2]{Healey-Lawler:N_sided_radial_SLE}}] 
\label{thm:independent-time-change}
Let $g_t \colon \bD \smallsetminus \bgamma[0,t] \to \bD$ be the Loewner chain of $\bgamma \in \commonpaths[T]$, and denote by 
$g_t^j \colon \bD \smallsetminus \gamma^j[0,t] \to \bD$ the mapping-out function of each $\gamma^j[0, t]$. 
Then, $g_{t, j} := g_t \circ (g^j_t)^{-1} \colon \bD \smallsetminus \hat{K}^j_t \to \bD$, 
where $\hat{K}^j_t := g^j_t \big(\bigcup_{i \neq j} \gamma^i[0,t]\big)$,
is a conformal bijection such that $g_{t, j}(0)=0$ and $g_{t, j}'(0)>0$.
Write $g^j_t(\gamma^j(t)) = \exp({\ii \xi^j_t})$, and $g_{t, j}(e^{\ii z}) =: e^{\ii h_{t, j}(z)}$ 
Then, we have
\begin{align}\label{eqn:indep-time-change}
\partial_t \indeptime_\bgamma^j(t) = \frac{1}{\big( h'_{t, j}(\xi^j_t) \big)^{2} }, \qquad  j \in \{1,\ldots,n\} , \; t \in [0,T] .
\end{align}
\end{lemA}

The proof of the above result in~\cite{Healey-Lawler:N_sided_radial_SLE} uses Loewner differential equation for $g_{t, j}$. 
However, the sequence of maps $(g_{t, j})_{t \ge 0}$ \emph{is not a Loewner chain}. Indeed, the maps $g_t$ and $g^j_t$ come from Loewner chains, 
but the map $g_{t, j} = g_t \circ (g^j_t)^{-1}$ first changes the domain and then applies a Loewner transform to the new domain. 
Taking the inverse $(g_{t,j})^{-1} = g^j_t\circ g_t^{-1}$ makes the problem more transparent: the map $h^{-1}_t$ generates a growing collection of curves, 
but $g^j_t$ maps out only one of the curves, deforming the other curves in a time-dependent way. 
For completeness, we give a short fix to the proof of \Cref{thm:independent-time-change}.

\begin{proof}[Proof of \Cref{thm:independent-time-change}]
Examining the proof in~\cite{Healey-Lawler:N_sided_radial_SLE}, the only property of $x \mapsto \partial_t h_{t, j}(x)$ used is that it is bounded away from $\pm\infty$ at $x=\xi^j_t$. 
This fact follows easily from the Cauchy integral formula, e.g., similarly as in the analysis detailed in~\cite[Proposition~3.12]{AHP:Large_deviations_of_DBM_and_multiradial_SLE}.
\end{proof}

The next result gives explicit bounds on the time-change from the common parameterization to the independent parameterization. 
It is interesting also on its own right: the time-change $\indeptime_\bgamma^j$ is uniformly comparable to $nt$
for any $j$ and any curves, with a correction for the lower bound that only depends on the number of curves. 

\begin{prop}\label{thm:time-change-properties-better}
The map $\bgamma \mapsto \bindeptime_\bgamma$ is continuous from $\commonpaths[T]$ to $L^\infty([0,T], [0,\infty)^n)$, and
\begin{align}\label{eqn:time-change-comparison}
nt - \tfrac{\log(n)}{2} \le \indeptime_\bgamma^j (t) < nt , \qquad n \geq 2, \; j \in \{1,\ldots,n\} , \; t \in [0,T] .
\end{align}
\end{prop}

The upper bound $\indeptime_\bgamma^j(t) = nt$ in~\eqref{eqn:time-change-comparison} would be achieved if all $n$ curves exactly coincided. 
The lower bound in~\eqref{eqn:time-change-comparison} is a consequence of the upper bound in \Cref{thm:h-derivative-bound-better}.

\begin{proof}
Consider a sequence $(\bgamma_k)_{k \in \bN}$ in $\commonpaths[T]$ converging to $\bgamma $ in $\commonpaths[T]$
Since these are simple disjoint curves, 
(recalling the definition~\eqref{def:toplogy-with-param} of the parameterized topology on $\commonpaths[T]$)
the corresponding inverse mapping-out functions 
also converge uniformly on compact subsets of $\bD$; thus their derivatives at the origin also converge, and hence the capacities converge:
\begin{align*}
\lim_{k \to \infty}\indeptime_{\bgamma_k}^j(t)
= \lim_{k \to \infty} \rcap( \gamma_k^j[0,t])
=  \rcap( \gamma^j[0,t])
= \indeptime_{\bgamma}^j(t) , \qquad j \in \{1,\ldots,n\} , \; t \in [0,T] .
\end{align*}

This shows the continuity of $\bgamma \mapsto \bindeptime_\bgamma$. 
The lower bound in~\eqref{eqn:time-change-comparison} is immediate from the monotonicity of capacity: 
$nt = \rcap( \bgamma[0,t]) \geq \rcap (\gamma^j [0,t]) = \indeptime^j(t) \geq 0$.
For the upper bound, we invoke \Cref{thm:h-derivative-bound-better} together with \Cref{thm:independent-time-change}, to obtain
\begin{align*}
\partial_t \indeptime_\bgamma^j(t) = \big( h'_{t, j}(\xi^j_t) \big)^{-2} 
\ge ( g'_{t,j}(0))^{2} = \Big(\tfrac{g_t'(0)}{(g^j_t)'(0)}\Big)^{2} = \exp\big(2nt - 2\indeptime_\bgamma^j(t)\big).
\end{align*}
Integrating this differential inequality with initial condition $\indeptime_\bgamma^j(0) = 0$ yields
\begin{align}\label{eqn:time-change-comparison-help}
2\indeptime_\bgamma^j(t) \ge \log(e^{2nt}+n \, -1 ) - \log(n) \ge 2nt - \log(n),
\end{align}
which gives the sought upper bound in~\eqref{eqn:time-change-comparison}.
\end{proof}

Using \Cref{thm:time-change-properties-better}, in the next result (\Cref{thm:uniform-lipschitz-time-change})
we will control the time-change function 
$\bindeptime_\bgamma = (\indeptime_{\bgamma}^j)_{j=1}^n \colon [0, T] \to [0, \infty)^n$
defined in~\eqref{eqn:def-of-time-change-sigma}. 
Throughout, we denote
\begin{align} \label{eq:Cdelta}
\commonpaths[T](\delta) := \Big\{ \bgamma \in \commonpaths[T] \cond \dist(\gamma^j,\gamma^i) > \delta \textnormal{ for all } 0 \le i < j \le n \Big\} \; \subset \; \commonpaths[T] , \qquad \delta > 0 .
\end{align}
We endow $\commonpaths[T](\delta)$ with the relative topology induced by $(\commonpaths[T],\dcommonpaths[T])$
with the metric~\eqref{def:toplogy-with-param}. 

\begin{lem}\label{thm:uniform-lipschitz-time-change}
Fix a starting point $\btheta_0 \in \chamber$. 
For each $\delta, \varepsilon > 0$, there exists a constant $L=L(\btheta_0,n,\delta, \varepsilon) \in [1,\infty)$ 
such that for any $\bgamma \in \commonpaths[T](\delta)$, $T \ge 0$, and $j \in \{1,\ldots, n\}$, the time-change $\indeptime_\bgamma^j$ is $L$-biLipschitz on $[0, T-\varepsilon]$.
\end{lem}

\begin{proof}
Without loss of generality, let $n \geq 2$. 
By \Cref{thm:time-change-properties-better}, for each $\bgamma \in \commonpaths[T]$ we have 
\begin{align*}
\indeptime_\bgamma^j (T) \ge nT - \tfrac{\log(n)}{2} \quad \xrightarrow{T \to \infty} \quad +\infty . 
\end{align*}
Hence, by Koebe distortion theorem (e.g.~\cite[Theorem~2.10.8]{AIM:Elliptic_partial_differential_equations_and_quasiconformal_mappings_in_the_plane}) and the definition~\eqref{eqn:def-of-time-change-sigma}, we have 
$\dist(0, \gamma^j[0, T]) \leq 4 \exp(-\rcap(\gamma^j[0,T])) = 4 \exp(-\indeptime_\bgamma^j(T)) \leq 4 e^{-T} \to 0$ as $T \to \infty$, 
so there exists $T_0 = T_0(\delta) > 0$ such that $\dist(\gamma^j[0, T], \gamma^i[0, T]) < \delta$ for every $0 \le i < j \le n$,
whence $\underset{T  > T_0}{\bigcup} \, \commonpaths[T](\delta) = \emptyset$.  
It thus suffices to show the claim for every $\bgamma \in \underset{T \in [0, T_0]}{\bigcup} \, \commonpaths[T](\delta)$.

Given $\bgamma \in \commonpaths[T](\delta)$ for $T \le T_0$, combining 
\Cref{thm:independent-time-change}, \Cref{thm:h-derivative-bound-better}~\&~\Cref{thm:time-change-properties-better}, 
\begin{align*}
e^{2 (1 - n) T_0}
\le \; & \exp \big(2\indeptime_\bgamma^j(t) - 2nt \big)
= \Big(\tfrac{(g^j_t)'(0)}{g_t'(0)}\Big)^{2} = \big( g_{t, j}'(0) \big)^{-2}
\le \big( h'_{t, j}(\xi^j_t) \big)^{-2} 
\\
= \partial_t\indeptime_\bgamma^j (t) 
\; \le \; &  \bigg( \tfrac{1}{4} \, \sin \Big(\pi\min\Big\{\hmeas{\bD\smallsetminus \hat{K}^j_t}(I^+_{\xi^j_t}; 0),\hmeas{\bD\smallsetminus \hat{K}^j_t}(I^-_{\xi^j_t}; 0)\Big\}\Big) \bigg)^{-2} 
=: H_\bgamma^j (t) 
, \qquad t \leq T .
\end{align*}
This already proves the Lipschitz lower bound $\indeptime_\bgamma^j(t)-\indeptime_\bgamma^j(s) \le e^{-2nT_0}(t-s)$ for $0 \le s < t \le T$. For each $k \in \bN$, take $\bgamma_k \in \commonpaths[T_k](\delta)$ such that $L^j_k := L^j_{\bgamma_k} \xrightarrow{k \to \infty} L^j$, where
\begin{align*}
L^j  = L^j(n,\delta, \varepsilon) := \sup_{T \in [0, T_0]}\sup_{\bgamma \in \commonpaths[T](\delta)} L_\bgamma^j , 
\qquad \textnormal{and} \qquad
L_\bgamma^j := \sup_{t \in [0, T-\varepsilon]} \partial_t \indeptime_\bgamma^j(t) .
\end{align*}
After passing to a subsequence, we may assume that each $(\gamma_k^i[0, T_k-\varepsilon])_{k \in \bN}$ converges to 
a compact set $\tilde K^i$ with respect to the Hausdorff metric, 
and to a hull $K^i$ with respect to the Carath\'eodory topology, and that $(\bgamma_k[0, T_k-\varepsilon])_{k \in \bN}$ converges to a hull $K$ in the Carath\'eodory topology. 
By~\cite[Lemma 2.1]{Abuzaid-Peltola:Large_deviations_of_radial_SLE0}, 
the domain $\bD\setminus K^i$ (resp.~$\bD\setminus K$) is the connected component of $\bD\setminus\tilde K^i$ (resp.~$\bD\setminus \bigcup_{j=1}^n \tilde K^j$) 
containing $0$, we have $K^i \subset K$ for every $1 \le i \le n$. 
Assuming that $K^i \not\subset K^j$ for any distinct $i, j \in \{1, \ldots, n\}$, we find that 
\begin{align*}
\dist(K^{i_1}, K^{i_2}) = \dist(\tilde K^{i_1}, \tilde K^{i_2}) = \lim_{k \to \infty}\dist(\gamma_k^{i_1}, \gamma_k^{i_2}) \ge \delta, \qquad 1 \le i_1 < i_2 \le n 
\end{align*}
(by continuity of $\dist(\cdot)$ with respect to the Hausdorff metric). 
In particular, $K^{i_1}$ and $K^{i_2}$ are disjoint, and consequently $K = \bigcup_{i=1}^n K^i$.
Let $J^\pm_{K^j}$ be the two components of $\partial\bD \smallsetminus K$ adjacent to $K^j$. 
By conformal invariance of the harmonic measure, for any $k \in \bN$ and $t \in [0, T]$ we have
\begin{align*}
H_{\bgamma_k}^j (t) 
\leq \bigg( \tfrac{1}{4} \, \sin \Big(\pi\min\Big\{\hmeas{\bD\smallsetminus K}(J^+_{K^j}; 0),\hmeas{\bD\smallsetminus K}(J^-_{K^j}; 0)\Big\}\Big) \bigg)^{-2}
=: H^j , 
\end{align*}
which shows that $\smash{L_k^j \xrightarrow{k \to \infty} L^j \le H^j = H^j(\btheta_0,\delta, \varepsilon)}$.
It thus remains to show that $K^i \not\subset K^j$ for any distinct $i, j \in \{1, \ldots, n\}$. 

Towards a contradiction, assume that $K^i \subset K^j$ for some $i \ne j$, so that in particular, the set $\tilde K^j$ disconnects $e^{\ii \theta_0^i}$ from the origin in $\bD$. 
Then, there exists points $z_k \in \bD$ and radii $r_k \searrow 0$ such that $\gamma^j_k[0, T_k-\varepsilon] \cup B_{z_k}(r_k)$ disconnects $e^{\ii\theta_0^i}$ from the origin in $\bD$; note that this implies that $\dist(z_k, \gamma_k^j[0, T_k-\varepsilon]) < r_k$. Denote by $K^j_k$ the hull whose complement is the component of $\bD\setminus(\gamma^j_k[0, T_k-\varepsilon] \cup B_{z_k}(r_k))$ containing the origin. 
Since $r_k \searrow 0$, we deduce the Carath\'eodory convergence $K^j_k \xrightarrow{k \to \infty} K^j$ of the hulls.

Passing to a further subsequence, we may assume that $(\gamma_k^i[0, T_k])_{k \in \bN}$ converges to a compact set $\tilde F^i \supset \tilde K^i$ in the Hausdorff metric, and that $(\bgamma^k[0, T_k-\varepsilon] \cup \gamma_k^i[0, T_k])_{k \in \bN}$ converges to a hull $F \supset K$ in the Carath\'eodory topology. Denote by $\boldsymbol\eta \in \commonpaths[\varepsilon]$ the mapped out curves $\boldsymbol\eta(t) := g_{\bgamma_k[0, T_k-\varepsilon]}(\bgamma(T_k-\varepsilon+t))$. By the Lipschitz lower bound, we have 
\begin{align*}
\rcap\big(\bgamma_k[0, T_k-\varepsilon] \cup \gamma_k^i[0, T_k]\big) - \rcap\big(\bgamma_k[0, T_k-\varepsilon]\big) 
= \rcap\big(\eta_k^i[0, \varepsilon]\big) 
= \indeptime_{\boldsymbol{\eta}_k}^i(\varepsilon) \ge e^{-2nT_0}\varepsilon > 0 ,
\end{align*}
uniformly on the shapes of the curves. Taking the limit $k \to \infty$ shows that $\rcap(F) > \rcap(K)$, so $F \ne K$.
Since $\tilde K^j \subset K^j \subset K$ for every $1 \le j \le n$, and by~\cite[Lemma 2.1]{Abuzaid-Peltola:Large_deviations_of_radial_SLE0} $\bD\setminus F$ is the component of $\bD\setminus(\bigcup_{j=1}^n \tilde K^j \cup \tilde F^i)$ containing origin, we must have $\tilde F^i \not\subset K$. Fix a point $z \in \tilde F^i\setminus K$, and denote $u := \dist(z, K) > 0$. By Carath\'eodory convergence $K^j_k \xrightarrow{k \to \infty} K^j \not\ni z$, for sufficiently large $k \in \bN$ we have $\dist(z, K^j_k) > \frac{u}{2}$, while by Hausdorff convergence we have $\dist(z, \gamma_k^i[0, T_k]) < \frac{u}{4}$ for sufficiently large $k \in \bN$. For large $k \in \bN$ for which we furthermore have $r_k < \frac{u}{4}$, 
the curve $\gamma_k^j[0, T_k]$ connects $e^{\ii \theta_0^i} \in K^j_k$ to $B_z(\frac{u}{4})$ disjoint from $K^j_k$. We conclude that $\gamma_k^i[0, T_k]$ must intersect the boundary of $K_k^j$, so
\begin{align*}
\gamma_k^i[0, T_k] \cap \big(\gamma_k^j[0, T_k-\varepsilon] \cup \overline{B_z(r_k)}\big) \supset \gamma_k^i[0, T_k] \cap \partial K^j_k \ne \emptyset.
\end{align*}
As $\gamma_k^i[0, T_k]$ and $\gamma_k^j[0, T_k -\varepsilon]$ are disjoint, there exists $w_k \in \gamma_k^i[0, T_k] \cap \overline{B_{z_k}(r_k)}\ne \emptyset$, and 
\begin{align*}
\dist(w_k, \gamma_k^j[0, T_k-\varepsilon]) \le \dist(w_k, z_k) + \dist(z_k, \gamma_k^j[0, T_k-\varepsilon]) \le 2r_k \xrightarrow{k \to \infty} 0.
\end{align*}
Since $w_k \in \gamma_k^i[0, T_k]$, this contradicts the assumption $\bgamma_k \in \commonpaths[T](\delta)$. 
By contradiction, we conclude that $K^i \not \subset K^j$ for any distinct $i, j \in \{1, \ldots, n\}$, which finishes the proof.
\end{proof}

\subsection{Topological lemmas for radial curves}
\label{subsec:topo}

Consider the collection of configurations of independently parametrized curves obtained by extending common-capacity-parameterized curves in $\commonpaths[T]$ to times $\bindepT = (\indepT^1,\ldots,\indepT^n)$ thus:
\begin{align*}
\Dextpath{T}{\bindepT} := \Big\{ \tilde\bgamma \in \Dindpaths[\bindepT] \cond 
\textnormal{$\tilde\bgamma|_{\bindeptime_{\bgamma}[0,T]} \circ \bindeptime_{\bgamma} = \bgamma$ for some $\bgamma \in \commonpaths[T]$ and $T \in [0,\infty)$}
\Big\} \; \subset \; \Dindpaths[\bindepT] .
\end{align*}
We define the projection $\prCommonTime{T} \colon \Dextpath{T}{\bindepT} \to \commonpaths[T]$ as $\prCommonTime{T}(\tilde\bgamma) := \bgamma$.
Recalling~\eqref{eq:Cdelta}, we denote
\begin{align*}
\Dextpath{T}{\bindepT}(\delta) := \Big\{ \tilde\bgamma \in \Dextpath{T}{\bindepT} \cond \prCommonTime{T}(\tilde\bgamma) \in \commonpaths[T](\delta) \Big\} \; \subset \; \Dindpaths[\bindepT]  , \qquad \delta > 0 .
\end{align*}
We endow these sets with the relative topology induced by $(\Dindpaths[\bindepT],\dindpaths[\bindepT])$
with the metric~\eqref{def:toplogy-inped-param}.

\begin{lem}\label{thm:convergence-of-common}
Fix $\bindepT \in [0,\infty)^n$ and $T, \delta > 0$, and let $(\tilde\bgamma_k)_{k \in \bN}$ be a sequence in $\Dextpath{T}{\bindepT}(\delta)$ converging to $\tilde\bgamma \in \Dindpaths[\bindepT]$. 
Then, $\tilde\bgamma \in \Dextpath{T}{\bindepT}$, and $\prCommonTime{T}(\tilde\bgamma_k)$ converges to $\prCommonTime{T}(\tilde\bgamma)$ in $\overline{\commonpaths[T](\delta)}$.
\end{lem}

\begin{proof}
By assumption, we have 
\begin{align*}
\dindpaths[\bindepT](\tilde\bgamma_k, \tilde\bgamma) 
= \sup_{(s^1,\ldots,s^n) \in [\boldsymbol{0},\bindepT]} \; \sum_{j=1}^n \big| \tilde\gamma_k^j (s^j) - \tilde\gamma^j(s^j) \big| \quad \xrightarrow{k \to \infty} \quad 0 ,
\end{align*}
the projections $\bgamma_k := \prCommonTime{T}(\tilde\bgamma_k) \in \commonpaths[T](\delta)$ form a sequence in $\commonpaths[T](\delta)$, 
and the time-changes $\bindeptime_k := \bindeptime_{\bgamma_k}$ form a uniformly bounded sequence in $L^\infty([0,T], [\boldsymbol{0},\bindepT])$. 
Since they are uniformly Lipschitz on every $[0, T-\varepsilon]$ with $\varepsilon > 0$ by \Cref{thm:uniform-lipschitz-time-change}, Arzel\`a-Ascoli theorem implies that
$\smash{\bindeptime_{k_i} \xrightarrow{i \to \infty} \bindeptime}$ along a subsequence, and the function 
$\bindeptime \colon [0,T] \to [\boldsymbol{0},\bindepT]$ is uniformly Lipschitz.
Now, we have
\begin{align*}
\sup_{t \in [0,T]}  \big| \gamma_k^j (t) - (\tilde\gamma^j \circ \indeptime_k^j)(t) \big| 
= \; & \sup_{t \in [0,T]} \big| (\tilde\gamma_k^j \circ \indeptime_k^j) (t) - (\tilde\gamma^j \circ \indeptime_k^j)(t) \big| \\
= \; & \sup_{s \in \indeptime_k^j[0,T]} \big| \tilde\gamma_k^j (s) - \tilde\gamma^j (s) \big| 
\le \sup_{s \in [0,\indepT^j]} \big| \tilde\gamma_k^j (s) - \tilde\gamma^j (s) \big| 
\quad \xrightarrow{k \to \infty} \quad 0 .
\end{align*}
Let $\boldsymbol\eta := \tilde\bgamma|_{\bindeptime[0,T]} \circ \bindeptime$ and
$\boldsymbol\eta_{k_i} := \tilde\bgamma|_{\bindeptime_{k_i}[0,T]} \circ \bindeptime_{k_i}$. Then, we see that 
\begin{align*}
\dcommonpaths[T](\bgamma_{k_i}, \boldsymbol\eta) 
\leq \sup_{t \in [0,T]} \; \sum_{j=1}^n \Big(
\big| \gamma_{k_i}^j (t) - \eta_{k_i}^j(t) \big| + \big| \eta_{k_i}^j(t) - \eta^j(t) \big| \Big)
\quad \xrightarrow{i \to \infty} \quad 0 ,
\end{align*}
so the projections $\bgamma_{k_i} := \prCommonTime{T}(\tilde\bgamma_{k_i})$ converge in the metric $\dcommonpaths[T]$ to $\boldsymbol\eta$.
It now remains to show that $\boldsymbol\eta \in \commonpaths[T](\delta)$
(indeed, by uniqueness of common parametrization, we then have $\boldsymbol\eta = \prCommonTime{T}(\tilde\bgamma)$, showing the uniqueness of the limit of the whole sequence $(\bgamma_k)_{k \in \bN}$).

Note that $\boldsymbol\eta$ (extended by, say, hyperbolic line segments to the origin) is a simple radial multichord 
(\Cref{def:radial_multichord}). 
We still need to show that the curves $\boldsymbol\eta$ have the common parameterization. 
To this end, note that the convergence $\smash{\bgamma_{k_i} \xrightarrow{i \to \infty} \boldsymbol\eta}$ in the metric $\dcommonpaths[T]$ implies that for each $t \in [0,T]$, the segments $\smash{\bgamma_{k_i}[0,t] \xrightarrow{i \to \infty} \boldsymbol\eta[0,t]}$ converge in the Carath\'eodory sense 
(e.g.,~by a disk version of~\cite[Proposition~6.3]{Kemppainen:SLE_book}), 
meaning that the inverse mapping-out functions $g_{\bgamma_{k_i}[0,t]}^{-1}(z)$ converge to $g_{\boldsymbol\eta[0,t]}^{-1}(z)$ uniformly on compact subsets of $[0,T] \times \bD \ni (t,z)$, and thus, the capacities converge uniformly in time: 
\begin{align*}
\rcap(\bgamma_{k_i}[0,t]) = - \log (g_{\bgamma_{k_i}[0,t]}^{-1})'(0) \quad \xrightarrow{i \to \infty}  \quad - \log (g_{\boldsymbol\eta[0,t]})'(0) = \rcap(\boldsymbol\eta[0,t]) ,
\end{align*}
and moreover, the maps $g_{\boldsymbol\eta[0,t]}$ solve the multi-slit Loewner equation~\eqref{eqn:multiradial_Loewner_1common} (cf.~\cite[Proposition~6.6]{Miller-Sheffield:QLE}),
where the driving function is the limit of the driving functions of $\bgamma_{k_i}$ and in particular, the curves in $\boldsymbol\eta$ have the common parameterization, as claimed. 
\end{proof}

\begin{lem}\label{thm:topological-lemma}
Fix $T > 0$ and $n \geq 2$. 
For any $\bindepT \in [0,\infty)^n$, the map $\prCommonTime{T} \colon \Dextpath{T}{\bindepT} \to \commonpaths[T]$ is continuous. 
If furthermore $\min\bindepT \ge nT$, then $\Dextpath{T}{\bindepT}$ is an open subset of $\Dindpaths[\bindepT]$. 
\end{lem}

\begin{proof}
The continuity of $\prCommonTime{T}$ on $\Dextpath{T}{\bindepT}$ is a direct consequence of \Cref{thm:convergence-of-common}. 
For the second claim, 
it suffices to show that each $\Dextpath{T}{\bindepT}(\delta) \subset \Dindpaths[\bindepT]$ is open.
To this end, take a sequence $(\tilde\bgamma_k)_{k \in \bN}$ in the complement $\Dindpaths[\bindepT] \smallsetminus \Dextpath{T}{\bindepT}(\delta)$ converging to $\tilde\bgamma \in \Dindpaths[\bindepT]$.
Define
\begin{align*}
T_k(u) := \sup \big\{t \ge 0 \cond \prCommonTime{t}(\tilde\bgamma_k) \in \commonpaths[t](u) \textnormal{ exists} \big\} , 
\qquad T(u) := \liminf_{k \to \infty}T_k(u) , \qquad u \ge 0.
\end{align*}
Note that $T(\delta) \le T$ by assumption, while for every $t \in [0, T(\delta))$, 
we have $\tilde\bgamma_k \in \Dextpath{t}{\bindepT}(\delta)$ for sufficiently large $k$, 
so that $\bgamma := \prCommonTime{T}(\tilde\bgamma) = \lim_{k \to \infty} \prCommonTime{T}(\tilde\bgamma_k) =: \lim_{k \to \infty} \bgamma_k$ by continuity of $\prCommonTime{T}$. 
The continuity of the time-change (\Cref{thm:time-change-properties-better}) then implies the uniform convergence $\smash{\bindeptime_k := \bindeptime_{\bgamma_k} \xrightarrow{k \to \infty} \bindeptime_{\bgamma} =: \bindeptime}$. 
Write $\bindeptime(T(\delta)) := \lim_{t \uparrow T(\delta)}\bindeptime(t)$.
Assuming that $\min\bindepT \ge nT$, by \Cref{thm:time-change-properties-better} we have $S := \min\bindepT - \underset{1 \le j \le n}{\max} \,\indeptime_{\bgamma}^j(T(\delta)) > 0$. 

We aim to show that for each $u > 0$ such that $\bgamma \in \commonpaths[T(\delta)](u)$, we have $u \ne \delta$.  
Set 
\begin{align*}
d := \underset{1 \le i < j \le n}{\min} \, \dist \big(\gamma^j[0, T(\delta)], \gamma^i[0, T(\delta)] \big) > u
\end{align*} 
and $\varepsilon := d-u > 0$. 
Fix $t \in [0, T(\delta))$ and $s \in (0, S)$ independent of $t$ such that
\begin{enumerate}[label=(\roman*)]
\item \label{aux_1}
$\diam \big(\tilde\gamma^j[\indeptime^j(t), \indeptime^j(t)+s] \big) < \frac{\varepsilon}{4}$ for all $1 \le j \le n$;

\item \label{aux_2}
we have $T_k(\delta) > t$, for every $k \in \bN$ large enough;

\item \label{aux_3}
we have $\indeptime_k^j(t) \le \indeptime^j(t) + s/2$, for every $k \in \bN$ large enough;

\item \label{aux_4}
$\dindpaths[\bindepT] \big(\tilde\gamma_k^j(\indeptime^j(t) + s), \tilde\gamma^j(\indeptime^j(t) + s) \big) < \frac{\varepsilon}{4}$ for all $1 \le j \le n$, for every $k \in \bN$ large enough.
\end{enumerate}
For such $k \in \bN$, we may thus estimate
\begin{align*}
\; &\dist \big(\tilde\gamma_k^j[0, \indeptime_k^j(t)+s/2], \tilde\gamma_k^i[0, \indeptime_k^i(t)+s/2] \big) &&\textnormal{[well-defined by~\ref{aux_2} and $s \in [0, S)$]}\\
\ge \; & \dist\big(\tilde\gamma_k^j[0, \indeptime^j(t)+s], (\tilde\gamma_k^i[0, \indeptime^i(t)+s]) \big) && \textnormal{[by~\ref{aux_3}]}\\
\ge \; & \dist\big(\tilde\gamma^j[0, \indeptime^j(t)+s], \tilde\gamma^i[0, \indeptime^i(t)+s] \big) - \tfrac{2\varepsilon}{4} && \textnormal{[by~\ref{aux_4}]}\\
\ge \; & \dist\big(\tilde\gamma^j[0, \indeptime^j(t)], \tilde\gamma^i[0, \indeptime^i(t)] \big) - \varepsilon && \textnormal{[by~\ref{aux_1}]}\\
\ge \; & \dist\big(\gamma^j[0, T(\delta)], \gamma^i[0, T(\delta)] \big) - \varepsilon &&\big[\textnormal{since } t < T(\delta) \implies \tilde \bgamma(\bindeptime[0, t]) \subset \bgamma[0, T(\delta)]\big]\\
> \; & d - \varepsilon \;=\; u.
\end{align*}
We thus have 
\begin{align*}
\frac{s}{2}
\leq \; & \underset{1 \le j \le n}{\max} \, \big| \indeptime_k^j(T_k(u)) - \indeptime_k^j(t) \big| 
\leq L(\btheta_0,n,u) \, |T_k(u) - t| \\
\qquad \Longrightarrow \qquad  
\underset{t \uparrow T(\delta)}{\lim} \; & \underset{k \to \infty}{\liminf} \, |T_k(u) - t| = \underset{t \uparrow T(\delta)}{\lim} \, |T(u) - t| = |T(u) - T(\delta)| > 0 ,
\end{align*}
which implies in particular that $u \ne \delta$ and thus, 
$\tilde\bgamma \notin \Dextpath{T(\delta)}{\bindepT}(\delta) \supset \Dextpath{T}{\bindepT}(\delta)$. 
This shows that $\Dindpaths[\bindepT] \smallsetminus \Dextpath{T}{\bindepT}(\delta)$ is indeed a closed subset, proving the second claim.
\end{proof}

\begin{remark}
The proof that $\Dextpath{T}{\bindepT} \subset \Dindpaths[\bindepT]$ is open in \Cref{thm:topological-lemma} uses the assumption $\min \bindepT \ge nT$ only to guarantee 
$S :=  \min \bindepT - \max_j\bindeptime_\bgamma^j(T(\delta)) > 0$, 
so that we can choose an element $s \in (0, S)$. 
The proof could be modified to hold for every $\bindepT > 0$, by extending each $\tilde\bgamma_k \in \Dextpath{T}{\bindepT}$ to, e.g, a curve $\hat\bgamma_k \in \Dindpaths[\bindepT+\boldsymbol{1}]$ such that $\lim_{k \to \infty} \hat\bgamma_k \in \Dindpaths[\bindepT+\boldsymbol{1}]$ exists, which ensures that $\hat S := \min \bindepT + 1 - \max_j\bindeptime_\bgamma^j(T(\delta)) \ge 1$, and then choosing $s \in (0, \hat S)$ instead. 
\end{remark}

\section{Finite-time LDP}
\label{sec:finite-time-LDP}

In this section, we will prove a finite-time LDP for multiradial $\SLE_{0+}$ in the topology of common-capacity-parameterized curves (\Cref{thm:n-radial-LDP-finite}).  
This result improves the finite-time multiradial LDP from~\cite{AHP:Large_deviations_of_DBM_and_multiradial_SLE} in two ways: first, the topology is stronger, and second, the role of the loop term is made explicit. 
Our aim is to use a generalized version of Varadhan's lemma (\Cref{thm:Varadhan-general} from Appendix~\ref{app:Varadhan}) 
to transfer the independent LDP to the $n$-radial case. 
However, there are two pitfalls we need to be careful with: differing time-parametrizations, and the restriction of the measures to non-colliding curves. 

\subsection{Brownian loop measure}

The Brownian loop measure~\cite{LSW:Conformal_restriction_the_chordal_case, Lawler-Werner:The_Brownian_loop_soup, Lawler:Conformally_invariant_processes_in_the_plane} is a sigma-finite measure on planar unrooted Brownian loops.
Fix $z \in \bD$ and $t > 0$. Consider the (sub-probability) measure $\mathbb{W}_z^t[\cdot]$ on Brownian paths on $\bD$, started from $z$ on the time interval $[0,t]$ and killed upon hitting the boundary $\partial\bD$. 
The disintegration of this measure with respect to the endpoint $w$ gives (sub-probability) measures $\mathbb{W}_{z \to w}^t[\cdot]$ on Brownian paths from $z$ to $w$ such that 
\begin{align*}
\mathbb{W}_z^t[\cdot] = \int_\bD \mathbb{W}^t_{z \to w}[\cdot] \ud w.
\end{align*}
The \emph{Brownian loop measure} on $\bD$ is defined as
\begin{align*}
\BLoop_{\bD}[\cdot] : = \int_0^{\infty} \frac{\ud t}{t} \int_\bD \mathbb{W}_{z \to z}^t[\cdot] \ud z .
\end{align*}
Upon forgetting the root $z$, this yields a measure on the set of unrooted and unparametrized loops in $\bD$, satisfying the following properties:
\begin{itemize}
\item \emph{Restriction property}: 
If $U \subset \bD$, then $\ud \BLoop_{U} [\ell]= \one\{\ell \subset U\} \, \ud \BLoop_\bD[\ell]$.
\item \emph{Conformal invariance}: 
If $\varphi \colon \bD \to D$ is a conformal bijection, then $\BLoop_D = \varphi_* (\BLoop_\bD)$.
\end{itemize}
The total mass of $\BLoop_\bD$ is infinite (e.g., because of small loops) but when considering only loops that intersect macroscopic disjoint closed sets, the measure is finite.
For disjoint closed subsets $\emptyset \neq K^j \subset \overline \bD$, with $j \in \{1,\ldots, n\}$, writing $\boldsymbol{K} = (K^1, \ldots,K^n)$, we set
\begin{align}\label{eq:BLM_term}
\cL(\boldsymbol{K}) :=
\sum_{j=2}^n \BLoop_{\bD} \big[ \ell \colon K^j \cap \ell \ne \emptyset \textnormal{ for at least two distinct } j \big] .
\end{align}
$\cL$ is positive, finite, conformally invariant, and continuous when viewed as a functional on the curve space $\commonpaths[T]$ with metric~\eqref{def:toplogy-with-param}. 
(\Cref{thm:BLM-cont} below is essentially~\cite[Lemma~3.2]{Peltola-Wang:LDP_of_multichordal_SLE_real_rational_functions_and_determinants_of_Laplacians}.)

\begin{lem}\label{thm:BLM-cont}
For each $T \in (0,\infty)$, the map $\cL$ is continuous from $\commonpaths[T]$ to $[0,\infty)$.
\end{lem}

\begin{proof}
Take a sequence $(\bgamma_k)_{k \in \bN}$ in $\commonpaths[T]$ converging to $\bgamma \in \commonpaths[T]$ as $k \to \infty$ in the metric~\eqref{def:toplogy-with-param}.
As in the proof of~\cite[Lemma~3.2]{Peltola-Wang:LDP_of_multichordal_SLE_real_rational_functions_and_determinants_of_Laplacians}, 
it suffices to show that for each $p \in \{ 2,\ldots,n \}$, 
\begin{align} \label{eqn: claim}
\; & |\cL(\bgamma) - \cL(\bgamma_k)|  
\leq 
\BLoop_{\bD} \big[ A^p \, \Delta \, A^p_k \big] 
\quad \overset{k \to \infty}{\longrightarrow} \quad 0 ,
\\
\qquad \textnormal{where} \qquad 
\nonumber
\; & A^p_k := \big\{ \ell \; \big| \; \ell \cap \gamma^i_k \neq \emptyset  \textnormal{ for at least $p$ of the } i \in \{1,\ldots,n\} \big\} , \\
\nonumber
\; & A^p := \big\{ \ell \; \big| \; \ell \cap \gamma^i \neq \emptyset  \textnormal{ for at least $p$ of the } i \in \{1,\ldots,n\} \big\} ,
\end{align}
and $A^p \, \Delta \, A^p_k := (A^p \smallsetminus A^p_k) \cup (A^p_k \smallsetminus A^p)$ is the symmetric difference.
To this end, consider a Brownian loop $\ell \in A^p \, \Delta \, A^p_k$. 
Then, either $\ell$ intersects less than $p$ of the curves $\bgamma_k$ and at least $p$ of the curves in $\bgamma$, or vice versa.
Without loss of generality, let us consider the former case. 
Then, since $p \geq 2$, the loop $\ell$ intersects at least two the curves in $\bgamma$, so in particular, $\ell$ is a macroscopic loop.
Furthermore, there exists an index $j \in \{1,\ldots,\np\}$ such that $\ell$ intersects $\gamma^j$ but not $\gamma^j_k$.
On the other hand, as $\gamma^j_k$ converges to $\gamma^j$, we see that when $k$ is large enough,  both $\gamma^j_k$ and $\gamma^j$ belong to a narrow tube.
But then, the total mass of $\ell$ intersecting $\gamma^j$ but avoiding $\gamma^j_k$ tends to zero when $k \to \infty$.  This proves~\eqref{eqn: claim}.
\end{proof}

\subsection{The $n$-radial $\SLE$ interaction term}

Let us now recall from~\cite{Healey-Lawler:N_sided_radial_SLE} how the $n$-radial $\SLE_\kappa$ measure (\Cref{def:n-radial_SLE_driving_functions})
is obtained from the product measure on $n$ independent $\SLE_\kappa$ curves by tilting by a Brownian loop measure term.
We denote by $\prob^\kappa$ the probability measure of $n$ independent radial $\SLE_\kappa$ curves in independent time,
and by $\Pr^\kappa$ the $n$-radial $\SLE_\kappa$ probability measure for the curves. 
For an almost surely finite stopping time $\tau$, we denote by $\Pr^\kappa_\tau$ the law of the $n$-radial $\SLE_\kappa$ curves restricted to the time-window $[0,\tau]$, 
(more precisely, the process restricted to its completed right-continuous filtration up to time $\tau$).
Using the time-change $\bindeptime$ from~\eqref{eqn:def-of-time-change-sigma}, 
we denote by $\probSLEindep$ the probability measure of $n$ independent radial $\SLE_\kappa$ curves in the common parameterization in the time-window $[0,T]$.
By~\cite[Propositions~3.4~\&~3.6]{Healey-Lawler:N_sided_radial_SLE}, 
\begin{align}\label{eqn:RNloop}
\frac{\ud \Pr^\kappa_T}{\ud \probSLEindep}(\bgamma) 
= \frac{\PartF(\btheta_T)}{\PartF(\btheta_0)} 
\exp \bigg( \hat\beta_n(\kappa) n T + \frac{\charge(\kappa) }{2} \cL(\bgamma[0,T]) \bigg) 
=: \exp\bigg(\frac{1}{\kappa}\Psi^\kappa_T(\bgamma)\bigg) ,
\end{align}
where $\btheta$ is the Loewner driving function of $\bgamma$, and $\PartF$ the partition function 
\begin{align}
\label{eq:PartF}
\PartF(\boldsymbol{\alpha}) := \; & \bigg( \underset{1\leq i<j\leq n}{\prod} \, \sin\bigg(\frac{\alpha^j-\alpha^i}{2}\bigg) \bigg)^{2/\kappa} ,
\end{align}
and $\cL$ is the Brownian loop term~\eqref{eq:BLM_term}, and the parameters are
\begin{align}\label{eqn:c-kappa-definition}
\hat\beta_n(\kappa) = \frac{(n-1) ((\kappa-4)^2+4 n)}{8 \kappa }
\qquad \textnormal{and} \qquad
\charge(\kappa) := \frac{(6-\kappa)(3\kappa-8)}{2\kappa} .
\end{align} 
Note that $\charge(\kappa)$ is the central charge (conformal anomaly number) of the associated conformal field theory,
and $\hat \beta_n(\kappa)$ is the $n$-interior scaling exponent.
For convenience, we also write
\begin{align} \label{eqn:psi-kappa-definition}
\Psi^\kappa_T(\bgamma) := \kappa \log\bigg(\frac{\PartF(\btheta_T)}{\PartF(\btheta_0)}\bigg) + \kappa\hat\beta_n(\kappa) n T + \frac{\kappa \charge(\kappa) }{2} \cL(\bgamma[0,T]) .
\end{align}

Let us now check that $(\Psi^\kappa_T)_{\kappa > 0}$ 
is a steadily converging family of maps, as required for the generalization of Varadhan's lemma 
(\Cref{thm:Varadhan-general} and \Cref{def:steady-convergence} in Appendix~\ref{app:Varadhan}). 
Its limit involves the semiclassical partition function
\begin{align*}
\LogPartF(\boldsymbol{\alpha}) 
:= - 2 \underset{1\leq i<j\leq n}{\sum} \log \sin\bigg(\frac{\alpha^j-\alpha^i}{2}\bigg)
= - \lim_{\kappa \to 0} \kappa \log \PartF(\boldsymbol{\alpha}) .
\end{align*}

\begin{lem}\label{thm:Psi-steady-convergence}
The maps $(\Psi^\kappa_T)_{\kappa > 0}$ converge steadily as $\kappa \to 0$ to the continuous function 
\begin{align}\label{eq:Psi-0}
\Psi^0_T \colon \commonpaths[T] \to \bR , \qquad 
\Psi^0_T(\bgamma) = \LogPartF(\btheta_0) - \LogPartF(\btheta_T) + \tfrac{1}{2} (n + 4) (n - 1) n \, T - 12 \cL(\bgamma[0,T]).
\end{align}
Furthermore, $\Psi^0_T$ is bounded from above.
\end{lem}

\begin{proof}
Because $\cL \geq 0$, $\Psi^0_T$ is bounded from above by $\Psi^0_T \le \LogPartF(\btheta_0) + \tfrac{1}{2} (n + 4) (n - 1) n T < \infty$. 
The continuity of $\Psi^0_T$ follows from \Cref{thm:BLM-cont}.
The steady convergence
$\smash{\Psi^\kappa_T \xrightarrow{\kappa \to 0} \Psi^0_T}$ follows easily from the 
definition~\eqref{eqn:psi-kappa-definition} of $\Psi^\kappa_T$ and \Cref{thm:steady-limit-addition} from Appendix~\ref{app:Varadhan}.
\end{proof}

\subsection{Finite-time LDP}

An LDP for the measures $\probSLEindep$ together with the above \Cref{thm:Psi-steady-convergence} 
and the generalized Varadhan's lemma (\Cref{thm:Varadhan-general}) would imply an LDP for the $n$-radial SLE-measures $\Pr^\kappa_T$. 
So, let us investigate how one could derive an LDP for $\probSLEindep$.
An LDP for the single-radial $\SLE_\kappa$-measures immediately yields an LDP for their independent product measure in any time window. 
Since by \Cref{thm:time-change-properties-better} the common time $T$ always occurs before the independent time $nT$, 
the set $\Dextpath{T}{\boldsymbol{nT}} \subset \Dindpaths[\boldsymbol{nT}]$ studied in \Cref{subsec:topo}
contains every curve in $\commonpaths[T]$ as sub-paths.
Hence, the measure $\probSLEindep$ of independent SLEs up to common time $T$ can be constructed by 
first, restricting the measure $\prob^\kappa_{\boldsymbol{nT}}$ to $\Dextpath{T}{\boldsymbol{nT}}$ 
--- let us denote the restricted measure by $\prob_{\bindeptime(T),\boldsymbol{nT}}^\kappa$ --- 
and then pushforwarding it through the projection $\prCommonTime{T} \colon \Dextpath{T}{\boldsymbol{nT}} \to \commonpaths[T]$ (also studied in \Cref{subsec:topo}).
One may hope to use the contraction principle in the last step; however, establishing an LDP for the measures 
$\prob_{\bindeptime(T),\boldsymbol{nT}}^\kappa$ turns out to be problematic. 
To see why, we need to investigate what happens to the LDP when restricting the measures.  
As a starting point, let us recall the case where an LDP can be trivially carried over through restrictions.

\begin{proposition}[{Restricted LDP; see, e.g.,~\cite[Lemma~4.1.5(b)]{Dembo-Zeitouni:Large_deviations_techniques_and_applications}}]\label{thm:restricted-LDP}
Let $X$ be a Hausdorff topological space, and $(P^\kappa)_{\kappa > 0}$ a family of probability measures on $X$ satisfying an LDP with good rate function $I \colon X \to [0, +\infty]$.  
Suppose $A \subset X$ is a measurable subset such that $I^{-1}[0, \infty) \subset A$ and $P^\kappa[A] = 1$ for every $\kappa > 0$. Then, the family $(P|_A^\kappa)_{\kappa > 0}$ of restricted measures satisfies an LDP with good rate function $I|_A \colon A \to [0, +\infty]$.
\end{proposition}

If we drop the assumption $I^{-1}[0, \infty) \subset A$, the above statement can be generalized to the case where $A$ is a closed continuity set for every $P^\kappa$ (i.e., $P^\kappa(\partial A) = 0$) such that $\liminf_{\kappa \to 0}\kappa\log P^\kappa(A) = 0$. 
The closedness of $A$ in this case is essential, because otherwise there is no guarantee for the goodness of the restricted rate function $I|_A$.

In our case, we are restricting to the open subset $\Dextpath{T}{\boldsymbol{nT}} \subset \Dindpaths[\boldsymbol{nT}]$ (\Cref{thm:topological-lemma}) 
which, although not straightforward to prove, should be a continuity set for every $\prob^\kappa_{\boldsymbol{nT}}$.  
To apply the generalized version of \Cref{thm:restricted-LDP}, one thus needs to instead restrict the measures $\prob^\kappa_{\boldsymbol{nT}}$ to the closure $(\DextpathCl{T}{\boldsymbol{nT}},\Dindpaths[\boldsymbol{nT}])$.  
To project the resulting measures to $\commonpaths[T]$ in a way compatible with the contraction principle, 
one would need to extend the projection $\prCommonTime{T}$ continuously to a map from $\DextpathCl{T}{\boldsymbol{nT}}$ to a completion of $\commonpaths[T]$.
Finally, one would need to extend $\Psi^\kappa_T$ continuously to a completion of $\commonpaths[T]$ in order for \Cref{thm:Varadhan-general} to be applicable. 

Although each step above may be doable, we can circumvent almost all of these technicalities by 
exchanging the order of the projection onto the common time and the change of measure from independent to $n$-radial $\SLE_\kappa$ measures. The only step we need to still carry out is a suitable continuous extension of $\Psi^0_T$, provided by the following result.

\begin{lem}\label{thm:tilde-Psi-continuous}
The mappings $\tilde \Psi^\kappa_T \colon \Dindpaths[\boldsymbol{nT}] \to [-\infty, \infty)$ defined as 
\begin{align*}
\tilde\Psi^\kappa_T(\tilde \bgamma) := 
\begin{cases}
( \Psi^\kappa_T\circ\prCommonTime{T} )(\tilde\bgamma) , & \tilde \bgamma \in \Dextpath{T}{\boldsymbol{nT}}, \\
-\infty, & \textnormal{otherwise},
\end{cases}
\end{align*}
converge steadily to $\tilde \Psi^0_T$ as $\kappa \to 0$. 
Furthermore, $\tilde \Psi^0_T$ is continuous and bounded from above.
\end{lem}

\begin{proof}
The steady convergence $\smash{\tilde\Psi^\kappa_T \xrightarrow{\kappa \to 0} \tilde\Psi^0_T}$ to a function bounded from above is a direct consequence of \Cref{thm:Psi-steady-convergence}. 
The continuity of $\tilde \Psi^0_T$ follows by showing that for every $x \in \bR$, 
the sets $\closed_{\le x} := (\tilde\Psi^0_T)^{-1}[-\infty, x]$ and $\closed_{\ge x} := (\tilde\Psi^0_T)^{-1}[x, +\infty)$ are closed subsets of $\Dindpaths[\boldsymbol{nT}]$.

To verify that $\closed_{\le x} \subset \Dindpaths[\boldsymbol{nT}]$ are closed, let $(\tilde \bgamma_k)_{k \in \bN}$ be a sequence in $F_{\le x}$ converging to $\tilde \bgamma \in \Dindpaths[\boldsymbol{nT}]$. 
Without loss of generality, assume $\tilde\Psi^0_T(\tilde\bgamma) > -\infty$, so that $\tilde \bgamma \in \Dextpath{T}{\boldsymbol{nT}}$. 
By \Cref{thm:topological-lemma}, as $\Dextpath{T}{\boldsymbol{nT}}$ is an open subset of $\Dindpaths[\boldsymbol{nT}]$, 
we have $\tilde\bgamma_k \in \Dextpath{T}{\boldsymbol{nT}}$ for sufficiently large $k \in \bN$. 
By the continuity of the projection $\prCommonTime{T}$ from \Cref{thm:convergence-of-common}, 
we conclude that 
\begin{align*}
\bgamma := \prCommonTime{T}(\tilde\bgamma) = \lim_{k \to \infty} \prCommonTime{T}(\tilde\bgamma_k) =: \lim_{k \to \infty} \bgamma_k. 
\end{align*} 
Since $\Psi^0_T$ is continuous, we thus obtain $\Psi^0_T(\bgamma) = \underset{k \to \infty}{\lim} \, \Psi^0_T(\bgamma_k) \le x$, so $\tilde \bgamma \in F_{\le x}$.

To verify that $\closed_{\ge x} \subset \Dindpaths[\boldsymbol{nT}]$ are closed, take a sequence $(\tilde \bgamma_k)_{k \in \bN}$ in $F_{\ge x} \subset\Dextpath{T}{\boldsymbol{nT}}$  converging to $\tilde \bgamma \in \Dindpaths[\boldsymbol{nT}]$, and write $\bgamma_k := \prCommonTime{T}(\tilde \bgamma_k)$. 
By compactness of the Hausdorff metric, we may pass to a subsequence $(\bgamma_{k_i})_{i \in \bN}$ such that $\gamma_{k_i}^j$ converges to some compact set $K^j \subset \overline \bD$ for each $j \in \{1,\ldots, n\}$.
Writing $\boldsymbol{K} = (K^1, \ldots,K^n)$, since the Brownian loop measure is continuous with respect to the Hausdorff metric by \Cref{thm:BLM-cont}, from~\eqref{eq:Psi-0} we get
\begin{align*}
-\infty < x \le \; & \Psi^0_T(\bgamma_{k_i}) 
\leq \LogPartF(\btheta_0) + \tfrac{1}{2} (n + 4) (n - 1) n \, T - 12 \cL(\bgamma_{k_i}[0,T])
\\
\quad \xrightarrow{i \to \infty} \quad
\; & \LogPartF(\btheta_0) + \tfrac{1}{2} (n + 4) (n - 1) n \, T - 12 \cL(\boldsymbol{K}) ,
\end{align*}
where $\cL(\boldsymbol{K}) < \infty$; 
so for large enough $i$, $\bgamma_{k_i} \in \commonpaths[T](\frac{\delta}{2})$, 
where $\delta := \min_{l < m} \dist(K^l, K^m) > 0$.
We thus find that 
$\smash{\bgamma_{k_i} \xrightarrow{i \to \infty} \prCommonTime{T}(\tilde\bgamma) =: \bgamma}$ and 
$\tilde\bgamma \in \Dextpath{T}{\boldsymbol{nT}}$ by \Cref{thm:convergence-of-common}.
Since $\Psi^0_T$ is continuous, we thus obtain $\Psi^0_T(\bgamma) = \underset{i \to \infty}{\lim}\,\Psi^0_T(\bgamma_{k_i}) \ge x$, so $\tilde \bgamma \in F_{\ge x}$.
\end{proof}

\begin{restatable}{thm}{LDPfinite}
\label{thm:n-radial-LDP-finite}
Fix $T \in (0,\infty)$. 
The family $(\Pr^\kappa_T)_{\kappa > 0}$ of laws of the $n$-radial $\SLE_\kappa$ curves 
satisfy an LDP on $\commonpaths[T]$ with good rate function $\nBessel_T$\textnormal{:}
for every open $G \subset \commonpaths[T]$ and closed $F \subset \commonpaths[T]$, we have
\begin{align}\label{eq:LDP_bounds_finite-time}
\begin{split}
\liminf_{\kappa \to 0+}\kappa\log\Pr^\kappa_T[G] \ge\;& 
-\inf_{\bgamma \in G}\nBessel_T(\bgamma) , \\
\limsup_{\kappa \to 0+}\kappa\log\Pr^\kappa_T[F] \le\;& 
-\inf_{\bgamma \in F}\nBessel_T(\bgamma).
\end{split}
\end{align}
The rate function $\nBessel_T \colon \commonpaths[T] \to [0,+\infty]$ can be written either in terms of the interaction~\eqref{eq:Psi-0}, 
\begin{align}\label{eqn:BM-loop-energy}
\nBessel_{T}(\bgamma)  := \nDindepenergy_{\bindeptime(T)}(\bgamma) - \Psi^0_T(\bgamma) ,
\qquad 
\textnormal{where}
\qquad 
\nDindepenergy_{\bindeptime(T)}(\bgamma) := \; & \sum_{j=1}^n \Denergy_{\indeptime_{\bgamma}^j(T)} \big(\gamma^j \circ (\indeptime_{\bgamma}^j)^{-1}\big) 
\end{align}
and $\Denergy_{\indeptime_{\bgamma}^j(T)}$ are the independent single-radial Loewner energies from Equation~\eqref{eqn:nDirichlet_energy_def}, or 
in terms of  the Dyson-Dirichlet energy of the multiradial Loewner driving function $\btheta$ of $\bgamma$,
\begin{align}\label{eq:Dyson-Dirichlet_energy}
\hspace*{-4mm}
\nBessel_T(\bgamma) 
:= \; &
\begin{dcases} 
\tfrac{1}{2} \int_0^T \sum_{j=1}^n \bigg| \tfrac{\ud}{\ud s} \theta^j_s - 2 \sum_{\substack{1 \leq k \leq n \\[.1em] k\neq j}} \cot \Big( \tfrac{\theta_s^j - \theta_s^k}{2} \Big) \bigg|^2 \ud s ,
\quad & \textnormal{if } \btheta 
\textnormal{ is abs.~cont.~on $[0,T]$,} \\
\infty , & \textnormal{otherwise.}
\end{dcases}
\end{align}
\end{restatable}

Note that $\tilde\bgamma \in \Dindpaths[\bindeptime(T)]$ comprising $\tilde\gamma^j := \gamma^j \circ (\indeptime_{\bgamma}^j)^{-1}$ are the curves corresponding to $\bgamma$
in the independent parameterization, and $\Denergy_{\indeptime_{\bgamma}^j(T)}(\tilde\gamma^j)$ are their single-radial energies. 

\begin{proof}
Consider the measure $\tilde\Pr^\kappa_T$ constructed by first tilting $\prob^\kappa_{\boldsymbol{nT}}$ by $\exp\big(\tfrac{1}{\kappa}\tilde\Psi^\kappa_T\big)$ from \Cref{thm:tilde-Psi-continuous} to obtain a measure $\tilde\Pr^\kappa_{\bindeptime(T),\boldsymbol{nT}}$, 
then restricting it to $\Dextpath{T}{\boldsymbol{nT}}$, 
and finally taking the pushforward through the projection $\prCommonTime{T}$. 
From this construction, we get\footnote{We denote by $\bE^\kappa_{\boldsymbol{t}}$ the expected value with respect to the measure $\prob^\kappa_{\boldsymbol{t}}$.} 
\begin{align*}
\tilde\Pr^\kappa_T[H] 
= \frac{\tilde\Pr^\kappa_{\bindeptime(T),\boldsymbol{nT}}\big[\prCommonTimeInv{T}H \big]}{\tilde\Pr^\kappa_{\bindeptime(T), \boldsymbol{nT}}\big[\Dextpath{T}{\boldsymbol{nT}}\big]} 
=\;& \frac{\bE^\kappa_{\boldsymbol{nT}}\bigg[\exp\Big(\frac{1}{\kappa}\tilde\Psi^\kappa_T(\tilde\bgamma)\Big)\one\{\tilde\bgamma \in \prCommonTimeInv{T}H\}\bigg]}{\bE^\kappa_{\boldsymbol{nT}}\bigg[\exp\Big(\frac{1}{\kappa}\tilde\Psi^\kappa_T(\tilde\bgamma)\Big)\one\{\tilde\bgamma \in \prCommonTimeInv{T}\commonpaths[T]\}\bigg]}\\
=\;& \frac{\bE^\kappa_{\bindeptime(T)}\bigg[\exp\Big(\frac{1}{\kappa}\Psi^\kappa_T(\bgamma)\Big)\one\{\bgamma \in H\}\bigg]}{\bE^\kappa_{\bindeptime(T)}\bigg[\exp\Big(\frac{1}{\kappa}\Psi^\kappa_T(\bgamma)\Big)\one\{\bgamma \in \commonpaths[T]\}\bigg]} = \Pr^\kappa_T[H],
\end{align*}
for every Borel set $H \subset \commonpaths[T]$; 
the measure $\tilde\Pr^\kappa_T$ thus coincides with the $n$-radial $\SLE_\kappa$ measure $\Pr^\kappa_T$.
\Cref{thm:tilde-Psi-continuous} gives the steady convergence $\smash{\tilde\Psi^\kappa_T \xrightarrow{\kappa \to 0} \tilde\Psi^0_T}$, and the limit $\tilde\Psi^0_T$ is a continuous function bounded from above. 
We can thus apply \Cref{thm:Varadhan-general} from Appendix~\ref{app:Varadhan} to transfer the LDP from the independent measures 
$\prob^\kappa_{\boldsymbol{nT}}$ to the measures $\tilde\Pr^\kappa_{\bindeptime(T), \boldsymbol{nT}}$, 
which have good rate function $\tilde \nBessel_T := \nDindepenergy_{\boldsymbol{nT}} - \tilde\Psi^0_T$. 
The assumptions of the restricted LDP (\Cref{thm:restricted-LDP} stated above) for the measures $\tilde\Pr_{\bindeptime(T),\boldsymbol{nT}}^\kappa$ 
are clearly satisfied, so the LDP is preserved upon restriction to $\Dextpath{T}{\boldsymbol{nT}}$. 
Finally, the LDP for $\Pr^{\kappa}_T = \tilde\Pr^\kappa_T = \tilde \Pr^{\kappa}_{\bindeptime(T), \boldsymbol{nT}} \circ \prCommonTimeInv{T}$ with good rate function $\nBessel_{T} := \nDindepenergy_{\bindeptime(T)} - \Psi^0_T$
follows from the contraction principle and the fact that 
$\nDindepenergy_{\boldsymbol{nT}}(\tilde\bgamma) - \tilde \Psi^0_T(\tilde\bgamma) \ge \nDindepenergy_{\bindeptime(T)}(\bgamma) - \Psi^0_T(\bgamma)$ holds for every $\tilde \bgamma \in \prCommonTimeInv{T}\{\bgamma\}$, 
with equality attained when each $\tilde\gamma^j$ is chosen to be the geodesic extension of $\gamma^j$ in the domain $\bD\smallsetminus\gamma^j$.

It remains to show that $\nBessel_{T}(\bgamma)$ also equals the Dyson-Dirichlet energy $\nBessel_{T}'(\btheta)$ given by the righthand side of Equation~\eqref{eq:Dyson-Dirichlet_energy}.
On the one hand, the contraction principle applied to the inverse Loewner transform $\mathscr{L}_{T}^{-1}$ sending curves $\bgamma$ to their multiradial driving functions $\btheta$ 
(which is continuous by the Loewner-Kufarev theorem, see, e.g.,~\cite[Theorem~8.5]{Berestycki-Norris:Lecture_notes_on_SLE})
shows that the multiradial Loewner driving process of the $n$-radial $\SLE_\kappa$ also satisfies an LDP with rate function $\nBessel_{T} \circ \mathscr{L}_{T}$.
On the other hand, from~\cite[Theorem~1.3]{AHP:Large_deviations_of_DBM_and_multiradial_SLE} 
we know that the multiradial Loewner driving process (namely, the $n$-radial Bessel process~\cite{Healey-Lawler:N_sided_radial_SLE}) satisfies an LDP with rate function $\nBessel_{T}'$. 
Hence, the uniqueness of the LDP rate function~\cite[Lemma~4.1.4)]{Dembo-Zeitouni:Large_deviations_techniques_and_applications} implies that $\nBessel_{T} \circ \mathscr{L}_{T} = \nBessel_{T}'$.
This shows that $\nBessel_{T}(\bgamma)$ also equals the righthand side of Equation~\eqref{eq:Dyson-Dirichlet_energy},
and finishes the proof.
\end{proof}

\section{Infinite time LDP and transience}
\label{sec:LDP}

In this section, we prove our main result: \Cref{thm:n-radial-LDP-infinite-time}. 
In order to
extend the finite-time LDP (\Cref{thm:n-radial-LDP-finite}) to infinite time, the key ingredients will be the escape probability estimates given in \Cref{thm:conditional-escape-probability} and \Cref{thm:escape-estimates-small-kappa}.
Alternative escape estimates (relating to the marginals of the measure) were very recently derived in~\cite{CHW:Multi-time_Loewner_energy_rate_function_for_large_deviation}, where an LDP was independently proven (though also relying on the results developed in~\cite{AHP:Large_deviations_of_DBM_and_multiradial_SLE, Abuzaid-Peltola:Large_deviations_of_radial_SLE0}).

\subsection{Escape estimates}

Denote $\bD_m := e^{-m}\bD$ for $m \in \bN$. 
For two integers $0<\u<\vee$ and for $j \in \{1,\ldots,n\}$, consider the escape events\footnote{Recall that we are working with the completed right-continuous filtration, so that the escape event $\xescapeevent{j}{\u}{\vee} \subset \commonpaths[\infty]$ is an open set for the common-capacity-parameterized curves.} 
$\xescapeevent{j}{\u}{\vee}$ that the curve $\gamma^j$ exits $\overline\bD_\u$ after entering $\bD_\vee$:
\begin{align}\label{eqn:escape-event}
\xescapeevent{j}{\u}{\vee} 
:= \; & \{\bgamma \in \commonpaths[\infty] \cond \gamma[\rhittingtime[\bgamma]^j(\vee), \infty) \not\subset \overline{\bD}_\u\} 
\; \subset \; \commonpaths[\infty] ,
\\
\label{eqn:escape-event-time}
\rhittingtime[\bgamma]^j(\vee)
:= \; & \inf\{t \ge 0 \colon |\gamma^j(t)| < e^{-\vee} \}.
\end{align}

The next result gives explicit deterministic bounds on the entry times $\rhittingtime[\bgamma]^j(\vee)$.
These bounds will be used in proving the escape estimates in \Cref{thm:escape-estimates-small-kappa}.

\begin{lem}\label{thm:rhittingtime-bounds-better}
The stopping times $\rhittingtime[\bgamma]^j(\vee)$ in~\eqref{eqn:escape-event-time} satisfy the deterministic bounds
\begin{align}\label{eqn:rhittingtime-bounds-better}
\vee - \log(4) \le n\rhittingtime[\bgamma]^j(\vee) \le \vee + \tfrac{\log(n)}{2} , \qquad \vee \ge 0 ,
\end{align}
uniformly for all curves $\bgamma \in \commonpaths[\infty]$, starting points $\btheta_0 \in \chamber$, and $j \in \{1,\ldots,n\}$.
\end{lem} 

\begin{proof}
Koebe $1/4$-theorem (e.g.~\cite[Theorem~2.10.6]{AIM:Elliptic_partial_differential_equations_and_quasiconformal_mappings_in_the_plane}) 
yields 
\begin{align*}
e^{-\vee} \le \big|\big( g^j_{\indeptime_\bgamma^j(\rhittingtime[\bgamma]^j(\vee))} \big)'(0)\big|^{-1} \le 4\, e^{-\vee} , \qquad e^{-\vee} = \dist \big(0, \gamma^j[0, \rhittingtime[\bgamma]^j(\vee)] \big) ,
\end{align*}
for the \emph{independently} parametrized Loewner chain $(g^j_s)_{s \ge 0}$ 
at time $s=\indeptime_\bgamma^j(\rhittingtime[\bgamma]^j(\vee))$.
By definition of the independent capacity-parametrization, we have 
$\big(g^j_s\big)'(0) = e^{s}$, so
\begin{align}\label{eqn:bounds-indep-hitting-time}
\vee - \log 4 \le \indeptime_\bgamma^j(\rhittingtime[\bgamma]^j(\vee)) \le \vee.
\end{align}
Recalling the bounds~\eqref{eqn:time-change-comparison} from \Cref{thm:time-change-properties-better} at time $t=\rhittingtime[\bgamma]^j(\vee)$ now yields~\eqref{eqn:rhittingtime-bounds-better}.  
\end{proof}

On the event $\xescapeevent{j}{\u}{\vee}$, at least one of the curves must enter $\overline{\bD}_\vee$ by time $\vee/n$, since
\begin{align}\label{eqn:min-rhittingtime-bound}
n \rhittingtime(\vee) = \rcap (\bgamma(\rhittingtime(\vee))) \leq \rcap(\overline{\bD} \smallsetminus \bD_\vee) =\vee,
\qquad 
\textnormal{where}
\quad
\rhittingtime(\vee):= \inf_{j\in\{1, \ldots, n\} }\rhittingtime[\bgamma]^j(\vee).
\end{align}
Together with the lower bound in~\Cref{thm:rhittingtime-bounds-better}, this implies a strong bound for $\rho(\vee)$:
\begin{align*}
\vee - \log 4 \le n\,\rhittingtime(\vee) \le \vee.
\end{align*}
In particular, $\rhittingtime(\vee)$ is deterministically contained in an interval of size $\frac{1}{n} \log 4$.

To better understand the infinite-time escape events $\xescapeevent{j}{\u}{\vee}$, we first consider $\Pr^\kappa_T (\xescapeevent{j}{\u}{\vee})$, the probability of escape in finite time $T>0$.
Recall that $n$-radial $\SLE_\kappa$ is obtained from independent $\SLE_\kappa$ with the common parameterization via the change of measure~\eqref{eqn:RNloop}:
\begin{align*}
\frac{\ud \Pr^\kappa_T}{\ud \probSLEindep}(\bgamma) 
= \frac{\PartF(\btheta_T)}{\PartF(\btheta_0)} \, 
\exp \bigg( \hat\beta_n(\kappa) n T + \frac{\charge(\kappa) }{2} \cL(\bgamma[0,T]) \bigg) 
=: \frac{\RNloop^\kappa_T}{\RNloop^\kappa_0} .
\end{align*}
From~\eqref{eqn:c-kappa-definition}, we see that $\charge(\kappa) \le 0$ when $\kappa \le 8/3$.
Since $\cL$ is positive, we find that
\begin{align} \label{eq:Mgle_bound}
\frac{\RNloop^\kappa_T}{\RNloop^\kappa_0} \leq \frac{\PartF(\btheta_T)}{\PartF(\btheta_0)}  \, e^{\hat\beta_n(\kappa) n T}, \qquad \kappa  \leq 8/3 .
\end{align} 
In particular, we readily obtain a finite-time escape estimate from
\begin{align} \label{eqn:escape-estimate-finite-time}
\Pr^\kappa_T [ \xescapeevent{j}{\u}{\vee} ] 
\leq \frac{\PartF(\btheta_T)}{\PartF(\btheta_0)} \, e^{\hat\beta_n(\kappa) n T} \, \probSLEindep [ \xescapeevent{j}{\u}{\vee} ] 
\leq \frac{\PartF(\btheta_T)}{\PartF(\btheta_0)} \, e^{\hat\beta_n(\kappa) n T} \, \prob^\kappa [ \xescapeevent{j}{\u}{\vee} ] ,
\end{align}
by applying the single curve result from~\cite[Theorem~4.6]{Abuzaid-Peltola:Large_deviations_of_radial_SLE0} to bound $\prob^\kappa [\xescapeevent{j}{\u}{\vee}]$ (see \Cref{thm:conditional-escape-probability}). 
However, we cannot directly obtain an infinite-time bound from~\eqref{eqn:escape-estimate-finite-time}, 
since the term $\exp\big( \hat\beta_n(\kappa) n T \big)$ tends to infinity as $T\to \infty$. 
Nevertheless, for small $\kappa$, an infinite-time escape estimate suitable for large deviation results can be obtained by decomposing $\xescapeevent{j}{\u}{\vee}$ into finite-time escape events 
(thus expressing $\Pr^\kappa_T [\xescapeevent{j}{\u}{\vee}]$ as a geometric series) and applying a certain conditional single-curve escape estimate,  as we will show next.

\begin{lem}\label{thm:conditional-escape-probability} 
There exist universal constants $C \in (0, \infty)$ and $\xi \in \bR$ such that the following holds. 
Fix $\kappa \in (0, 4]$, $\vee > 0$, a simply connected domain $D \subsetneq \bC$ containing $\bD_\vee$, and $x \in \partial D$.
Denote\footnote{We use the evident definitions for $\gamma \sim \prob^\kappa_{(D;x,0)}$ analogous to (\ref{eqn:escape-event},~\ref{eqn:escape-event-time}).} 
by $\prob^\kappa_{(D;x,0)}$ the law of radial $\SLE_\kappa$ on $D$ from $x$ to $0$. 
Then, we have
\begin{align}\label{eqn:conditional-escape-estimate}
\prob^\kappa_{(D;x,0)} \Big[ \gamma[\rhittingtime[\gamma](\vee), \infty) \not\subset \overline{\bD}_\u \Big] 
\le C \,\exp\Big(\tfrac{\xi - (\vee-\u)(8-\kappa)}{2\kappa}\Big), \qquad 0 < \u < \vee-1 .
\end{align}
\end{lem}

\begin{proof}
Combining~\cite[Proposition~4.7]{Abuzaid-Peltola:Large_deviations_of_radial_SLE0} with Brownian excursion measure estimates similar to those in the proof of radial case of~\cite[Theorem~4.6]{Abuzaid-Peltola:Large_deviations_of_radial_SLE0} 
shows that the lefthand side of~\eqref{eqn:conditional-escape-estimate} is upper bounded by a universal constant times $\hat c^{8/\kappa}\exp\big(\tfrac{2\zeta - (\vee-\u)(8-\kappa)}{2\kappa}\big)$, where
$\hat c \in (0, \infty)$ and $\zeta \in \bR$ are universal constants coming from the last equation of the proof of~\cite[Theorem~4.6]{Abuzaid-Peltola:Large_deviations_of_radial_SLE0}. 
Choosing $\xi = 16\log \hat c + 2\zeta$ gives the desired bound.
\end{proof}

\begin{prop}\label{thm:escape-estimates-small-kappa}
There exists a constant $\kappa_0 = \kappa_0(n) > 0$ depending only on $n$ such that the following holds.
For each $\u \ge 1$, $\varepsilon \in (0,1)$, and $\M \in [0,\infty)$, there exists $\vee > \u$ such that
\begin{align}\label{eq:escape-estimates-small-kappa}
\Pr^\kappa [ \xescapeevent{j}{\u}{\vee} ] \leq \varepsilon \, e^{- \frac{1}{\kappa} \M} , \qquad \kappa \in (0,\kappa_0) , \; j \in \{1,\ldots,n\}.
\end{align}
\end{prop}

\cite[Theorem~4.6]{Abuzaid-Peltola:Large_deviations_of_radial_SLE0} gives the case $n=1$ with $\kappa_0 = 4$.
The number $\vee = \vee(\btheta_0,\u,\varepsilon,\M,n)$ in~\eqref{eq:escape-estimates-small-kappa} depends on $\u$, $\varepsilon$, $\M$, and $n$, as well as on the starting point $\btheta_0 \in \chamber$ of the evolution. 

\begin{proof}
Write $l_n = \frac{\log(n)}{2}$ and $v_n = l_n + \log(4) + n + 2$.
First, let us fix $0 < \u < \vee - v_n + 1$. 
For each $T \ge 0$, define the event
\begin{align}\label{eq:xescapeevent_T}
\xescapeevent{j}{\u}{\vee}(T)
:= \; & \{\bgamma \in \commonpaths[\infty] \cond \gamma^j[\rhittingtime[\bgamma]^j(\vee), T] \not\subset \overline{\bD}_\u\}.
\end{align}
Note that the events $\xescapeevent{j}{\u}{\vee}(T)$ are $\Pr^\kappa_T$-measurable and satisfy $\xescapeevent{j}{\u}{\vee} = \bigcup_{T \ge 0}\xescapeevent{j}{\u}{\vee}(T)$. By monotone convergence, we may thus decompose
\begin{align*}
\Pr^\kappa [ \xescapeevent{j}{\u}{\vee} ] 
= \sum_{\mindex = 0}^\infty \Pr^\kappa_{\mindex+1} \big[ \xescapeevent{j}{\u}{\vee}(\mindex+1) \smallsetminus \xescapeevent{j}{\u}{\vee}(\mindex) \big] .
\end{align*}
By~\Cref{thm:rhittingtime-bounds-better}, the time when the $j^{th}$ curve enters level $\vee$ satisfies
$\rhittingtime[\bgamma]^j(\vee) \ge \lfloor\frac{\vee - \log(4)}{n}\rfloor =: \mindexMin$, so $\xescapeevent{j}{\u}{\vee}(\mindex+1) = \emptyset$ for all $\mindex < \mindexMin$, and the corresponding summands vanish.  
For $\mindex \ge \mindexMin$, 
on the complement of the event $\xescapeevent{j}{\u}{\vee}(\mindex)$, we have $\gamma^j[\rhittingtime[\bgamma]^j(\vee), \mindex] \subset \overline\bD_\u$, while by~\Cref{thm:rhittingtime-bounds-better} we have 
$\mindex \ge \rhittingtime[\bgamma]^j(n\mindex - l_n)$. 
On the event $\xescapeevent{j}{\u}{\vee}(\mindex+1)$, this implies that
\begin{align*}
A_\mindex^j := \big\{ \gamma^j\big[\rhittingtime[\bgamma]^j(n\mindex - l_n), \mindex+1 \big] \not\subset \overline\bD_\u \big\}
\end{align*} 
occurs, so we obtain
\begin{align}\label{eqn:escape-probability-decomposition-2}
\Pr^\kappa [ \xescapeevent{j}{\u}{\vee} ] 
\le\;& \sum_{\mindex=\mindexMin}^\infty \Pr^\kappa_{\mindex+1} [ A_\mindex^j ] .
\end{align}
By \Cref{thm:rhittingtime-bounds-better}, we have 
\begin{align*}
\underset{1 \le j \le n}{\min} \,\rhittingtime[\bgamma]^j(n\mindex - l_n) 
\geq \mindex - \frac{l_n + \log(4)}{n} := \tau_\mindex ,
\end{align*}
so $\bD_{\mindex} \subset \bD\smallsetminus\boldsymbol\eta$ for every $\boldsymbol \eta \in \commonpaths[\tau_\mindex]$. 
Estimating the Radon-Nikodym derivative as in~\eqref{eq:Mgle_bound} with $\kappa  \leq 8/3$ and using the tower property, we obtain the bound
\begin{align*}
\Pr^\kappa_{\mindex+1} [A_\mindex^j] 
= \; & 
\bE^\kappa_{\bindeptime(\mindex+1)} \bigg[ \bE^\kappa_{\bindeptime(\mindex+1)} \bigg[ \frac{\RNloop^\kappa_{\mindex+1}}{\RNloop^\kappa_0} \, \one_{A_\mindex^j} \cond \mathcal{F}_{\tau_\mindex} \bigg] \bigg] \\
= \; & 
\bE^\kappa_{\bindeptime(\mindex+1)} \bigg[ \frac{\RNloop^\kappa_{\tau_\mindex}}{\RNloop^\kappa_0} \, \bE^\kappa_{\bindeptime(\mindex+1)} \bigg[ \frac{\RNloop^\kappa_{\mindex+1}}{\RNloop^\kappa_{\tau_\mindex}} \, \one_{A_\mindex^j} \cond \mathcal{F}_{\tau_\mindex} \bigg] \bigg]
\\
\leq \; & e^{\hat\beta_n(\kappa) n(\mindex+1-\tau_\mindex)} \, \Ex^\kappa \bigg[ \frac{\PartF(\btheta_{\mindex+1})}{\PartF(\btheta_{\tau_\mindex})} \, \bE^\kappa \big[ \one_{A_\mindex^j} \cond \mathcal{F}_{\tau_\mindex} \big] \bigg].
&& \textnormal{[by~\eqref{eq:Mgle_bound}]}
\end{align*}
On the one hand, by the domain Markov property of radial $\SLE_\kappa$ we can estimate the conditional probability on the righthand side using \Cref{thm:conditional-escape-probability}:
as $0 < \u < n\mindex - l_n - 1$ (here, we need $0 < \u < \vee - v_n + 1$), 
\begin{align*}
\bE^\kappa \big[ \one_{A_\mindex^j} \cond \mathcal{F}_{\tau_\mindex} \big]
\leq \bE^\kappa \Big[ \one_{\{ \gamma^j [\rhittingtime[\bgamma]^j(n\mindex - l_n), \infty ) \not\subset \overline\bD_\u \}} \cond \mathcal{F}_{\tau_\mindex} \Big]
\leq \; & C \,\exp\Big(\tfrac{\xi - (n\mindex - l_n -\u)(8-\kappa)}{2\kappa}\Big) . 
\end{align*}
On the other hand, by \Cref{lem:pf_estimate} stated below and because $\PartF \leq 1$, for $\kappa \leq 4$ we have 
\begin{align}\label{eq:pf_estimate_pf}
\Ex^\kappa \bigg[ \frac{\PartF(\btheta_{\mindex+1})}{\PartF(\btheta_{\tau_\mindex})} \bigg]
\leq \Ex^\kappa \bigg[ \frac{1}{\PartF(\btheta_{\tau_\mindex})} \bigg]
\leq \frac{\exp \Big( \frac{n(n^2-1)}{12} \, \tau_\mindex \Big) }{\PartF(\btheta_{0})} . 
\end{align}
Collecting these estimates into~\eqref{eqn:escape-probability-decomposition-2}, we obtain
\begin{align*}
\Pr^\kappa [ \xescapeevent{j}{\u}{\vee} ] 
\le\;& \sum_{\mindex=\mindexMin}^\infty \Pr^\kappa_{\mindex+1} [ A_\mindex^j ] 
\le \frac{C \, e^{ \frac{1}{\kappa} \A(n,\u) }}{\PartF(\btheta_{0})} \,
\sum_{\mindex=\mindexMin}^\infty e^{ - \frac{1}{\kappa} \mindex \B(n,\kappa) } , \qquad \kappa \leq \kappa_0 = \kappa_0(n) := \frac{1}{n^2} ,
\end{align*}
where (recalling the formula~\eqref{eqn:c-kappa-definition} for $\hat\beta_n(\kappa)$)
\begin{align*}
\A(n, \u) = \; &  \frac{\xi}{2} + 4(l_n+\u) + \frac{(n-1) (4 + n)}{2} (n + l_n + \log(4)) ,
\\
\B(n,\kappa) = \; & n \bigg( \frac{8-\kappa}{2} - \frac{\kappa(n^2-1)}{12} \bigg) \geq 0 , \qquad \kappa \leq \kappa_0(n) := \frac{1}{n^2} ,
\end{align*}
and the function $\kappa \mapsto \B(n,\kappa)$ is decreasing and strictly positive in the interval $\kappa \in (0,\kappa_0(n)]$. 
Estimating the geometric series as 
\begin{align}\label{eq:geometric_series}
\sum_{\mindex=\mindexMin}^\infty e^{- \frac{1}{\kappa} \mindex \B(n,\kappa)} 
= \frac{e^{-\frac{1}{\kappa} \mindexMin \B(n,\kappa)} }{1-e^{-\frac{1}{\kappa} \B(n,\kappa)} } 
\le \frac{e^{\frac{1}{\kappa} \B(n,\kappa)}  }{1-e^{-\frac{1}{\kappa} \B(n,\kappa)} } \, e^{- \frac{1}{\kappa}  \big( \frac{\vee - \log(4)}{n} \big) \B(n,\kappa)} , 
\end{align}
noting that $\kappa \mapsto (1-e^{-\frac{1}{\kappa} \B(n,\kappa)})^{-1}$ is increasing in $\kappa$, 
and using the formulas~(\ref{eq:LogPartF},~\ref{eq:PartF}) for the partition functions $\LogPartF = -\kappa \log \PartF$, 
with $C_n = C \, (1-\exp(-\frac{1}{\kappa_0} \B(n,\kappa_0)))^{-1}$, we see that
\begin{align*}
\Pr^\kappa [ \xescapeevent{j}{\u}{\vee} ] 
\le\;& C_n \, \exp \bigg(\frac{1}{\kappa} \Big( \A(n,\u) + \B(n,\kappa) + \LogPartF(\btheta_{0}) 
+ \frac{\log(4)}{n} \, \B(n,\kappa) - \frac{\vee}{n} \, \B(n,\kappa) \Big) \bigg)
\leq \varepsilon \, e^{- \frac{1}{\kappa} \M} 
\end{align*}
if we choose 
$\vee \geq \vee_0 \wedge v_n$ when $C_n \leq \varepsilon$
and 
$\vee \geq (\vee_0  \wedge v_n) + \frac{n\, \log C_n}{\B(n,\kappa_0)} - \frac{n \, \log \varepsilon }{\B(n,\kappa_0)}$ when $C_n > \varepsilon$, where
\begin{align*}
\vee_0 = \vee_0(\btheta_0,\u,\varepsilon,\M,n) := \; & 
n + \log(4)  + \frac{n \, \A(n,\u)}{\B(n,\kappa_0)} + \frac{n \, \LogPartF(\btheta_{0}) }{\B(n,\kappa_0)} 
+ \frac{n \, \M}{\B(n,\kappa_0)} .
\end{align*}
This concludes the proof.
\end{proof}

\begin{lem}\label{lem:pf_estimate}
For any $t \geq 0$, we have
\begin{align} \label{eq:partF_boud}
\Ex^\kappa \bigg[ \frac{1}{\PartF(\btheta_{t})}  \bigg] 
\leq \frac{\exp \Big( \frac{n(n^2-1)}{12} \, t \Big)}{\PartF(\btheta_{0})}  , \qquad \kappa \leq 4 .
\end{align}
\end{lem}

\begin{proof}
By~\cite[Corollary~3.4]{AHP:Large_deviations_of_DBM_and_multiradial_SLE}, the driving functions for $n$-radial $\SLE_\kappa$ with the common parameterization (\Cref{def:n-radial_SLE_driving_functions}) 
are obtained from tilting the measure of independent standard Brownian motions
by the martingale from~\cite[Corollary~2.5]{AHP:Large_deviations_of_DBM_and_multiradial_SLE} with $\mathcal{E} = (\PartF)^2$:
\begin{align*}
M_t := \big(\PartF(\btheta_t)\big)^2 \exp \bigg( \! - \frac12 \int_0^t \frac{\Delta \big( \PartF(\btheta_s) \big)^2 }{\big(\PartF(\btheta_s) \big)^2 } \ud s \bigg) .
\end{align*}
Therefore, we have 
\begin{align*}
\Ex^\kappa \bigg[ \bigg( \frac{\PartF(\btheta_{0}) }{\PartF(\btheta_{t})} \bigg)^2 \bigg]
= \mathbb{W}_{t} \bigg[ \frac{M_{t}}{M_0} \, \bigg( \frac{\PartF(\btheta_{0}) }{\PartF(\btheta_{t})} \bigg)^2 \bigg]
= \mathbb{W}_{t} \Bigg[  \exp\bigg( \! - \frac12 \int_0^{t} \frac{\Delta \big( \PartF(\btheta_s) \big)^2 }{\big(\PartF(\btheta_s) \big)^2 } \ud s \bigg) \Bigg].
\end{align*}
Using~\cite[Lemma 5.1]{Healey-Lawler:N_sided_radial_SLE}, we find that 
\begin{align*}
- \frac12\frac{\Delta \big( \PartF(\btheta_s) \big)^2 }{\big(\PartF(\btheta_s) \big)^2 } 
= \; &
- \frac{2}{\kappa} \sum_{j=1}^n \Bigg( \sum_{\substack{1 \leq i \leq n \\[.1em] i\neq j}} \cot \bigg( \frac{\theta^j_t-\theta^i_t}{2}\bigg) \Bigg)^2
+ \frac12 \sum_{j=1}^n \sum_{\substack{1 \leq i \leq n \\[.1em] i \neq j}}  \csc^2 \bigg( \frac{ \theta^j_t - \theta^i_t }{2} \bigg) 
\\
= \; & \frac{\kappa-4}{2\kappa} \sum_{j=1}^n \Bigg( \sum_{\substack{1 \leq i \leq n \\[.1em] i\neq j}} \cot \bigg( \frac{\theta^j_t-\theta^i_t}{2}\bigg) \Bigg)^2 + \frac{n(n^2-1)}{6} 
\\
\leq \; & \frac{n(n^2-1)}{6}  , \qquad \kappa \leq 4 .
\end{align*}
Using the Cauchy-Schwarz inequality, 
we obtain the estimate~\eqref{eq:partF_boud}.
\end{proof}

\subsection{Transience}

As a consequence of the estimate in \Cref{thm:escape-estimates-small-kappa}, we obtain immediately that $n$-radial $\SLE_\kappa$ curves with small values of $\kappa$ are transient.
In fact, if we allow the multiplicative constant in the escape estimate~\eqref{eq:escape-estimates-small-kappa}
to depend on $\kappa$, we obtain transience for all $\kappa \in (0,8/3]$. 

\begin{prop}\label{thm:escape-estimates}
Fix $\kappa \in (0,8/3]$ and $n \ge 2$. 
There exists a constant $C_n(\kappa) \in (0,\infty)$ depending only on $\kappa$ and $n$ such that for 
$\vee > 2n + 2\log(4) +  \frac{\log(n)}{2}$, we have 
\begin{align*}
\Pr^\kappa [ \xescapeevent{j}{\u}{\vee} ]  \le C_n(\kappa) \, \exp\Big(\tfrac{(\u-\vee)(8-\kappa)}{2\kappa}\Big) , \qquad
0 < \u < \vee - \tfrac{\log(n)}{2}  - \log(4) - n - 1 
, \; j \in \{1,\ldots,n\}.
\end{align*}
\end{prop}

The same argument also gives an estimate for the $\SLE_\kappa$ curves intersecting a crosscut. 

\begin{proof}
In the proof of \Cref{thm:escape-estimates-small-kappa}, instead of using the estimate~\eqref{eq:pf_estimate_pf},
we know that by~\cite[Proposition~5.7]{Healey-Lawler:N_sided_radial_SLE}, there exists a constant $\hat C(\kappa) = \hat C_n(\kappa) > 0$ such that the density for $(\btheta_t)_{t \ge 0}$ is bounded from above by $\hat C(\kappa) \, (\PartF(\cdot))^4$ for every $t > 1$. Hence, we have
\begin{align*}
\Ex^\kappa \bigg[ \frac{1}{\PartF(\btheta_{\tau_\mindex})} \bigg]
\leq \hat C(\kappa) \int_{\chamber} (\PartF(\boldsymbol{\alpha}))^3 \ud \boldsymbol{\alpha} =: \tilde C(\kappa) < \infty , 
\qquad \textnormal{since} \qquad \tau_\mindex > 1 .
\end{align*}
Thus, collecting the estimates in the proof of \Cref{thm:escape-estimates-small-kappa}, we obtain
\begin{align*}
\Pr^\kappa [ \xescapeevent{j}{\u}{\vee} ] 
\lesssim \;& \tilde C(\kappa) \, 
\sum_{\mindex=\mindexMin}^\infty
e^{\hat\beta_n(\kappa) n(\mindex+1-\tau_\mindex)} \,
\exp\Big(\tfrac{\xi - (n\mindex - l_n -\u)(8-\kappa)}{2\kappa}\Big) 
\\
\lesssim \;& C_n'(\kappa) \, \tilde C(\kappa) \, \exp\Big(\tfrac{\xi}{2\kappa}\Big) \, 
 \exp\Big(\tfrac{l_n(8-\kappa)}{2\kappa}\Big) \, \exp\Big(\tfrac{\u(8-\kappa)}{2\kappa}\Big)
\sum_{\mindex=\mindexMin}^\infty \exp\Big(\!- \mindex n \tfrac{(8-\kappa)}{2\kappa}\Big) ,
\end{align*}
where $C_n'(\kappa) = \exp\Big(\hat\beta_n(\kappa) n \big(1+\frac{l_n + \log(4)}{n} \big) \Big)$.
Estimating the geometric series as 
\begin{align*}
\sum_{\mindex=\mindexMin}^\infty \exp\Big(\!- \mindex n \tfrac{(8-\kappa)}{2\kappa}\Big)
\le \frac{\exp\Big(\!- \frac{\vee(8-\kappa)}{2\kappa}\Big) \exp\Big(\frac{(8-\kappa) \log(4)}{2\kappa}\Big) }{1 - \exp\Big(\!- \tfrac{n (8-\kappa)}{2\kappa}\Big)} 
, \qquad \mindexMin \le \frac{\vee - \log(4)}{n} ,
\end{align*}
yields the desired estimate. 
\end{proof}

\Transience*

\begin{proof}
Set $\mindexMin := \frac{\log(n)}{2} + \log(4) + n + 1$. 
\Cref{thm:escape-estimates} with $\u=\mindex$ and $\vee = 2\mindex$ gives
\begin{align*}
\sum_{\mindex = \mindexMin}^\infty \Pr^\kappa [ \xescapeevent{j}{\mindex}{2\mindex} ]
\le C_n(\kappa) \, \sum_{\mindex = \mindexMin}^\infty \exp\Big(\tfrac{-\mindex(8-\kappa)}{2\kappa}\Big) 
< \infty .
\end{align*}
By the convergent Borel-Cantelli lemma, almost surely only finitely many of the events $\xescapeevent{j}{\mindex}{2\mindex}$ occur:
there exists an almost surely finite constant $\Mindex$ such that $\gamma^j[\rhittingtime[\bgamma]^j(2\mindex), \infty) \subset \bD_\mindex$ for $\mindex \ge \Mindex$ and $1 \le j \le n$.
Taking $\mindex \to \infty$, we see that $\underset{t \to \infty} {\lim} \, \gamma^j(t) = 0$ almost surely.
\end{proof}

\subsection{Infinite-time LDP}

We are now ready to begin the proof of the infinite-time LDP.
Similarly to the single-radial case in~\cite{Abuzaid-Peltola:Large_deviations_of_radial_SLE0}, the idea is to first pass to the projective limit
\begin{align*}
\limcommonpaths
= \Big\{ \cev\bgamma := \big( \bgamma|_{[0,T]} \big)_{T \ge 0} \condbig \bgamma|_{[0,s]} = \prDpaths{s,t} (\bgamma|_{[0,t]}) \textnormal{ for all } 0 < s \le t \Big\} 
\; \subset \; \prod_{T \ge 0}\commonpaths[T] ,
\end{align*} 
invoke Dawson-G\"artner theorem, and then apply the escape estimate from \Cref{thm:escape-estimates-small-kappa} together with \Cref{thm:continuity-sets} stated below, 
in order to find closed homeomorphic sets under the inclusion $\limembedding \colon \commonpaths[\infty] \to \limcommonpaths$,
sending $\bgamma \mapsto \cev\bgamma$, 
whose probabilities tend to $1$ with arbitrarily good exponential rates as $\kappa \to 0$. 
A generalized version of the contraction principle can then be used to improve the topology to that for the common-capacity-parameterized curves in $\commonpaths[\infty]$.

The topology on $\limcommonpaths$ is induced from the product topology on $\prod_{T\ge 0}\commonpaths[T]$ 
by the projections $\prDpaths{T} (\cev\bgamma) =: \bgamma|_{[0,T]}$, 
and coincides with the topology of uniform convergence on compacts and thus ``forgets'' about the tail behavior of the curves. 
To make the inclusion $\limembedding$ to become a homeomorphism, we can restrict to subsets of $F \subset \commonpaths[\infty]$ where the tails do not escape:

\begin{lem}\label{thm:continuity-sets}
If for every $\u \in \bN$, there exists $\vee(\u) \in \bN$ such that 
$F \subset \bigcap_{j=1}^n (\commonpaths[\infty]\smallsetminus \xescapeevent{j}{\u}{\vee(\u)})$, 
then the inclusion $\limembedding|_{F}$ restricted to $F \subset \commonpaths[\infty]$ is a homeomorphism onto its image.
\end{lem}

\begin{proof}
We follow the proof of~\cite[Lemma~4.3]{Abuzaid-Peltola:Large_deviations_of_radial_SLE0} (special case of~\Cref{thm:continuity-sets} with $n=1$).
Since $\limembedding$ is continuous, it suffices to show that $\limembedding|_F \colon F \to \limembedding(F)$ is an open map. 
For curves $\bgamma \in \commonpaths[\infty] \smallsetminus \xescapeevent{j}{\u}{\vee}$, 
we have $\bgamma[\rhittingtime[\bgamma]^j(\vee), \infty) \subset \overline{\bD}_\u$, while by~\Cref{thm:rhittingtime-bounds-better}, we have $\rhittingtime[\bgamma]^j(\vee) < \vee$ for sufficiently large $\vee > 0$.  
Thus, for any $\bgamma \in F$ and $\boldsymbol{\eta} \in \prDpaths{\vee(\u)}^{-1}(B_{\dcommonpaths[\vee(\u)]}(\prDpaths{\vee(\u)}(\cev\bgamma),e^{-\u})) \cap F$, 
\begin{align*}
\dcommonpaths(\bgamma, \boldsymbol{\eta}) \le \max \Big\{ \dcommonpaths[\vee(\u)] \big( \prDpaths{\vee(\u)}(\cev\bgamma) , \prDpaths{\vee(\u)}(\cev{\boldsymbol{\eta}}) \big) , \, \diam(\overline{\bD}_\u) \Big\}
= 2e^{-\u} .
\end{align*}
Since the projection $\prDpaths{\vee(\u)} \circ \limembedding^{-1} \colon \limcommonpaths \to \commonpaths[\vee(\u)]$ is continuous, 
$\limembedding \big(\prDpaths{\vee(\u)}^{-1}(B_{\dcommonpaths[\vee(\u)]}(\prDpaths{\vee(\u)}(\cev\bgamma),e^{-\u})) \cap F\big)$ 
is an open subset of $\limembedding(F)$. As $\u \in \bN$ is arbitrary, we conclude that $\limembedding|_F$ is an open map.
\end{proof}

We obtain our main \Cref{thm:n-radial-LDP-infinite-time} by following the proof strategy of~\cite[Theorem~1.2]{Abuzaid-Peltola:Large_deviations_of_radial_SLE0}.

\LDPinfinite*

\begin{proof} 
By the finite-time result from \Cref{thm:n-radial-LDP-finite} and Dawson-G\"artner theorem 
(see~\cite[Theorem~4.6.1]{Dembo-Zeitouni:Large_deviations_techniques_and_applications} and~\cite[Lemma~F]{Abuzaid-Peltola:Large_deviations_of_radial_SLE0}), 
the $n$-radial SLE measures $\limPr^\kappa := \lim_{T\to\infty}\Pr^\kappa_T\circ\prDpaths{T}$ on the projective limit $\limcommonpaths$ satisfy an LDP with good rate function 
\begin{align*}
\limnBessel := \sup_{T \geq 0}\nBessel_T\circ\prDpaths{T} \colon \limcommonpaths \to [0, +\infty].
\end{align*} 
As in the proof of~\cite[Theorem~1.2]{Abuzaid-Peltola:Large_deviations_of_radial_SLE0}, 
we can apply a generalized contraction principle (see~\cite[Lemma~C.1 in Appendix~C]{Abuzaid-Peltola:Large_deviations_of_radial_SLE0}) 
to the (discontinuous) map $\limembedding^{-1} \colon \limcommonpaths \to \commonpaths[\infty]$ to transfer the LDP to the $n$-radial SLE measures $\Pr^\kappa = \limPr^\kappa\circ\limembedding$ on the space $\commonpaths[\infty]$ of common-capacity-parameterized curves.
To this end, we need to check that the assumptions in~\cite[Lemma~C.1(i) in Appendix~C]{Abuzaid-Peltola:Large_deviations_of_radial_SLE0} are satisfied. Fix $M \in [0,\infty)$. 
By~\Cref{thm:escape-estimates-small-kappa}, there exists $\kappa_0 > 0$ and a sequence $\vec \vee = (\vee(\u))_{\u \in \bN}$ of integers such that
\begin{align*}
\Pr^\kappa [\xescapeevent{j}{\u}{\vee(\u)}] \le 2^{-\u}e^{-M/\kappa}, \qquad  \u \in \bN, \, \kappa \in (0,\kappa_0] , \; j \in \{1,\ldots,n\}.
\end{align*}
By~\Cref{thm:continuity-sets}, the image of the set $F_{\vec \vee} := \bigcap_{u \in \bN}\bigcap_{j=1}^n(\commonpaths[\infty] \smallsetminus \xescapeevent{j}{\u}{\vee(\u)})$ under $\limembedding$ is a closed continuity set of $\limembedding^{-1}$. 
The union bound yields
\begin{align*}
\limPr^\kappa [\limcommonpaths \smallsetminus \limembedding(F_{\vec \vee})] 
= \Pr^\kappa [\commonpaths[\infty] \smallsetminus F_{\vec \vee}] 
\le \sum_{\u \in \bN}\sum_{j=1}^n \Pr^\kappa [\xescapeevent{j}{\u}{\vee(\u)}] \le n \, e^{-M/\kappa},
\end{align*}
so in particular, we have
\begin{align*}
\lim_{\kappa \to 0} \kappa\log \limPr^\kappa [\limcommonpaths \smallsetminus \limembedding(F_{\vec \vee})]  \le -M.
\end{align*}
This is the needed assumption for~\cite[Lemma~C.1]{Abuzaid-Peltola:Large_deviations_of_radial_SLE0} to conclude the proof of the LDP.
Identity~\eqref{eq:Dyson-Dirichlet_energy_infty} follows from the fact that the rate function $\nBessel_T$ is increasing in $T$ and agrees with the Dyson-Dirichlet energy~\eqref{eq:Dyson-Dirichlet_energy} 
for all finite times, by \Cref{thm:n-radial-LDP-finite}.
\end{proof}

\subsection{The multiradial Loewner energy}
\label{subsec:multiradial_Loewner_energy}

To derive the Brownian loop formula~\eqref{eq:BLM_energy} for the $n$-radial Loewner energy, 
we use the behavior of finite-energy systems in terms of the semiclassical partition function~\eqref{eq:LogPartF}:

\minimum*
\begin{proof}
Taking $T \to \infty$ in~\eqref{eq:Dyson-Dirichlet_energy}, we see that $\btheta$ has a finite Dyson-Dirichlet energy~\eqref{eq:Dyson-Dirichlet_energy_infty}. 
By~\cite[Proposition 4.7]{AHP:Large_deviations_of_DBM_and_multiradial_SLE}, the translated driving function $\btheta_T - \theta^1_T$ has the limit 
\begin{align*}
\lim_{T \to \infty} \Big( \btheta_T - (\theta^1_T, \ldots, \theta^1_T) \Big) =  \big(0, \tfrac{2\pi}{n}, \ldots, \tfrac{2(n-1)\pi}{n} \big) .
\end{align*}
Since the function $\LogPartF$ defined in~\eqref{eq:LogPartF} is translation-invariant, we obtain
\begin{align*}
\lim_{T \to \infty}\LogPartF(\btheta_T) =  -2\sum_{1 \le i < j \le n} \log\sin\big(\tfrac{\pi(j-i)}{n}\big).
\end{align*}
It remains to check that this is the minimum of $\LogPartF$. 
By~\cite[Proposition 4.4]{AHP:Large_deviations_of_DBM_and_multiradial_SLE}, the gradient flows $\bvartheta$ of the potential $\LogPartF$ satisfying
\begin{align*}
\tfrac{\ud}{\ud t} \bvartheta_t = -\nabla \LogPartF(\bvartheta_t), \qquad \bvartheta_0 \in \chamber ,
\end{align*}
stabilize to some $\boldsymbol{\zeta} = \boldsymbol{\zeta}(\bvartheta_0) = (\zeta, \zeta+\frac{\pi}{n}, \ldots, \zeta+\frac{(n-1)2\pi}{n})$ in the limit $t \to \infty$. 
Since $t \mapsto \LogPartF(\bvartheta_t)$ is a non-increasing function, we conclude that 
\begin{align*}
\LogPartF(\bvartheta_0) \; \ge \;  \LogPartF(\boldsymbol{\zeta}) = -2\sum_{1 \le i < j \le n} \log\sin\big(\tfrac{\pi(j-i)}{n}\big).
\end{align*}
This finishes the proof. 
\end{proof}

\BLMForm* 

\begin{proof}
From the proof of \Cref{thm:n-radial-LDP-infinite-time}, we have
\begin{align*}
\nBessel(\bgamma) 
\; = \;\; & \sup_{T \geq 0} \nBessel_T(\bgamma)
\; = \; \sup_{T \geq 0} \Big( \nDindepenergy_{\bindeptime(T)}(\bgamma) - \Psi^0_T(\bgamma) \Big) 
&& \textnormal{[by~\eqref{eqn:BM-loop-energy}]}
\\
\; = \;\; & \limsup_{T \to \infty} \Big( \nDindepenergy_{\bindeptime(T)}(\bgamma) - \Psi^0_T(\bgamma) \Big) 
\\
\; = \;\; &  \nDindepenergy(\bgamma) \; - \; \LogPartF(\btheta_0) \; + \; \limsup_{T \to \infty} \Big( \LogPartF(\btheta_T) \; + \; 12 \cL(\bgamma[0,T]) \; - \; \tfrac{1}{2} (n + 4) (n - 1) n \, T \Big) ,
\end{align*}
and using the notation~\eqref{eq:cLrenorm}, we claim that
\begin{align}\label{eq:claim_LE}
\nBessel(\bgamma) = \nDindepenergy(\bgamma) -\LogPartF(\btheta_0) + \inf_{\btheta \in \chamber}\LogPartF(\btheta) + \limsup_{T \to \infty} 12 \cLrenorm(\bgamma[0,T]) 
=: \hat\nBessel(\bgamma) .
\end{align}
To verify the identity~\eqref{eq:claim_LE}, we consider the infinite-time behavior of $\btheta_T$.
If
\begin{align*} 
U_\infty := \liminf_{T \to \infty} \LogPartF(\btheta_T) = \inf_{\btheta \in \chamber}\LogPartF(\btheta) =: u_n ,
\end{align*} 
then~\eqref{eq:claim_LE} clearly holds. Since $u_n$ coincides with the value of $\LogPartF$ on equally-spaced configurations by \Cref{thm:minimum}, 
we see that if $U \neq u_n$, then~\cite[Proposition~4.7]{AHP:Large_deviations_of_DBM_and_multiradial_SLE} implies that $\nBessel(\bgamma) = +\infty$, 
so we see that if $U < \infty$, then also $\hat\nBessel(\bgamma) = +\infty = \nBessel(\bgamma)$, giving~\eqref{eq:claim_LE}.

Thus, let us assume that $U = +\infty$.
If $\nDindepenergy(\bgamma) = +\infty$, then $\hat\nBessel(\bgamma) = +\infty = \nBessel(\bgamma)$, giving~\eqref{eq:claim_LE}.
It remains to consider the case where $\nDindepenergy(\bgamma) < \infty$.
Then, the curves in $\bgamma$ are simple and do not touch the boundary $\partial \bD$ except at their starting points, and the divergence in $\nBessel(\bgamma) = +\infty$ is due to their interaction. 
We need to argue that also $\hat\nBessel(\bgamma) = +\infty$, which follows from arguing that $\cLrenorm(\bgamma[0,T]) \xrightarrow{T \to \infty} \infty$. 
The assumption $U = +\infty$ implies that
\begin{align*} 
\lim_{T \to \infty} \delta(\btheta_T) = 0 , \qquad \textnormal{where} \qquad \delta(\btheta_T) := \min_{1 \le j \le n} \big| \theta_T^{j+1}-\theta_T^j \big| . 
\end{align*} 
By the restriction property and conformal invariance of Brownian loop measure, we have
\begin{align*}
\cL(\bgamma[0,t+1]) - \cL(\bgamma[0,t]) \ge \cL(g_t(\bgamma[t,t+1])) , \qquad t \ge 0.
\end{align*}
Let $M := 1+\tfrac{(n + 4) (n - 1)}{24} \, n$. 
By \Cref{thm:loop-blowup} (proven below) applied to $s=1$ and to the curve $g_t(\bgamma[t,t+1])$, there exists $\varepsilon = \varepsilon(s,\nDindepenergy(\bgamma),M)> 0$
such that $\delta(\btheta_t) \leq \varepsilon$ and $\cL(g_t(\bgamma[t,t+1])) \geq M$ for all $t \ge T_\varepsilon$. Thus, for any $m \in \bN$ we have
\begin{align*} 
\cLrenorm(\bgamma[0,t+m]) 
\ge \; & \cL(\bgamma[0,t]) - \tfrac{(n + 4) (n - 1)}{24} \, n (t+m) + \sum_{j=1}^m \underbrace{\cL(g_{t+j-1}(\bgamma[t+j-1,t+j]))}_{\; \geq M}
\\
\ge \; & m - \tfrac{(n + 4) (n - 1)}{24} \, n t, \qquad t \ge T_\varepsilon.
\end{align*} 
Taking $m \to \infty$ shows that $\cLrenorm(\bgamma[0,T]) \xrightarrow{T \to \infty} \infty$ and concludes the proof. 
\end{proof}

\begin{lem}\label{thm:loop-blowup}
For every $s > 0$ and $M, M' \in [0,\infty)$, there exists $\varepsilon = \varepsilon(s,M,M')> 0$ such that if $\bgamma \in \commonpaths[s]$ satisfies $\nDindepenergy_{\bindeptime(s)}(\bgamma) \le M'$ and $\btheta_0 \in \chamber$ satisfies $\delta(\btheta_0) < \varepsilon$, then $\cL(\bgamma[0,s]) \ge M$.
\end{lem}

\begin{proof}
For each $\epsilon > 0$, let $L(s,M',\epsilon) := \underset{\boldsymbol{\eta} \in \mathfrak C^n_s(M';\epsilon)}{\inf} \,\cL(\boldsymbol{\eta}[0,s])$, where 
\begin{align*}
\mathfrak C^n_s(M';\epsilon) := \big\{\bgamma \in \commonpaths[s] \cond \delta(\btheta_0) < \epsilon \textnormal{ and } \nDindepenergy_{\bindeptime(s)}(\bgamma) \le M'\big\},
\end{align*}
Fix a decreasing sequence $\epsilon_k \downarrow 0$ as $k \uparrow \infty$, 
and a sequence of $\bgamma_k \in \mathfrak C^n_s(M';\epsilon_k)$ such that $\cL(\bgamma_k[0,s]) \le L(s,M',\epsilon_k) + \epsilon_k$.
After passing to a subsequence, we may assume that $\bgamma_k$ converges to a $n$-tuple of compact sets $\boldsymbol{K} = (K^1, \ldots, K^n)$ in the Hausdorff metric. 
By the continuity of the Brownian loop measure (\Cref{thm:BLM-cont}) $\cL(\bgamma_k[0,s]) \to \cL(\boldsymbol{K})$ as $k \to \infty$. 
It suffices to prove that $\cL(\boldsymbol{K}) = +\infty$: indeed, this ensures the existence of an integer $k \in \bN$ 
such that $\cL(\bgamma_k[0, s]) \ge M + \epsilon_k$, so that $L(s,M',\epsilon_k) \ge \cL(\bgamma_k[0, s]) - \epsilon_k \ge M$.

Let $\tilde\bgamma_k \in \Dextpath{s}{ns}$ be the independent geodesic extension of $\bgamma_k$. 
Then, the $n$-dimensional Dirichlet energy~\eqref{eqn:nDirichlet_energy_def} is finite:
\begin{align*}
\nDenergy_{\boldsymbol{ns}}(\tilde\bgamma) 
:= \sum_{j=1}^n \frac{1}{2} \int_0^{ns} \big| \tfrac{\ud}{\ud t} \tilde\vartheta_t^j \big|^2 \ud t
= \sum_{j=1}^n \Denergy_{ns}(\tilde\gamma^j) 
= \nDindepenergy_{\bindeptime(s)}(\bgamma) \leq M' < \infty ,
\end{align*}
(where $\tilde\vartheta^j$ are the independent Loewner driving functions of $\tilde\gamma^j$), 
\begin{align*}
\sum_{j=1}^n \Denergy_{ns}(\tilde\gamma^j) = \nDindepenergy_{\bindeptime(s)}(\bgamma) \leq M' < \infty ,
\end{align*}
and because the radial Loewner energy $\Denergy_{ns}$ is a good rate function, 
thus possessing compact sub-level sets~\cite[Theorem~1.2]{Abuzaid-Peltola:Large_deviations_of_radial_SLE0}, 
after passing to a further subsequence, we may assume that $\tilde\bgamma_k \to \tilde\bgamma \in \Dindpaths[ns]$ as $k \to \infty$. Note that $K^j \subset \tilde \bgamma^j$ for each $j$, so $K^1, \ldots, K^n$ are simple curves having independent capacities bounded uniformly from below as 
\begin{align*}
2\rcap(K^j) = \; & \lim_{k \to \infty} 2\indeptime_{\bgamma_k}^j(s) \ge \log(e^{2ns}+n-1)-\log(n) > 0 .
&& \textnormal{[by~\eqref{eqn:time-change-comparison-help}]}
\end{align*}
Note that $\cL(\boldsymbol{K}) = +\infty$ follows by showing that $K^i \cap K^j \neq \emptyset$ for some $1 \le i < j \le n$.
Towards a contradiction, suppose $K^i \cap K^j = \emptyset$ for every $1 \le i < j \le n$, and write $K := \bigcup_{j=1}^n K_j$. 
Since each $K^j$ is a simple curve, the components of $\partial\bD \setminus \bigcup_{j = 1}^n K$ are exactly the open boundary segments $I^j$ between the starting points of $K^j$ and $K^{j+1}$. 
As each $I^j$ is accessible from the origin, $\hmeas{\bD\setminus K}(I^j;0) > 0$. 
The conformal invariance and domain continuity of the harmonic measure now yields the sought contradiction:
\begin{align*}
0 = \lim_{k \to \infty} \epsilon_k \ge \lim_{k \to \infty} \delta(\btheta_k) \ge \min_{1 \le j \le n} \hmeas{\bD\setminus K}(I^j;0) > 0 .
\end{align*}
This finishes the proof.
\end{proof}

As a consequence, we obtain the explicit asymptotics of the Brownian loop measure interaction term for finite-energy radial multichords, 
which is linear in the capacity-time (and coincides with a perhaps surprising choice of a cocycle for the Virasoro algebra~\cite{Gelfand-Fuchs:Cohomologies_of_the_Lie_algebra_of_vector_fields_on_the_circle}).

\BLM*

\begin{proof}
This is a direct consequence of the formula~\eqref{eq:BLM_energy} in \Cref{thm:BLMForm}.
\end{proof}

\appendix
\section{Transfer of LDPs via absolute continuity}
\label{app:Varadhan}

As in~\cite{Peltola-Wang:LDP_of_multichordal_SLE_real_rational_functions_and_determinants_of_Laplacians, AHP:Large_deviations_of_DBM_and_multiradial_SLE}, 
in the present work we shall transfer the LDP from independent measures to the multicurve measure via absolute continuity using Varadhan's lemma. Since this is a recurring strategy, we find it useful to encapsulate it into a general framework. 

We first recall the classical statement, emphasizing the change of measure perspective.

\begin{lemA}
[Varadhan's lemma; see, e.g.,~\cite{Dembo-Zeitouni:Large_deviations_techniques_and_applications}, Lemmas~4.3.4 and 4.3.6]
\label{lemma:Varadhan}

Suppose that the probability measures $(P^\kappa)_{\kappa>0}$ satisfy an LDP in a topological space $X$
with good rate function $E$. 
Let $\Phi \colon X \to \bR$ be a function bounded from above, and let $Q^\kappa$ be a measure on $X$ absolutely continuous with respect to $P^\kappa$ with Radon-Nikodym derivative 
\begin{align*}
\frac{\ud Q^\kappa}{\ud P^\kappa}(x) = \exp \Big(\frac{1}{\kappa} \Phi(x) \Big) , \qquad x \in X . 
\end{align*}
Then, the following hold.
\begin{enumerate}
\item \label{item1_Var} 
If $\Phi$ is upper semicontinuous, then for any closed subset $F$ of $X$, 
\begin{align*}
\limsup_{\kappa\to 0+} \kappa \log Q^\kappa[F]
\leq - \inf_{x \in F} \big(E(x) - \Phi(x)\big) .
\end{align*}
\item \label{item2_Var} 
If $\Phi$ is lower semicontinuous, then for any open subset $O$ of $X$,
\begin{align*}
\liminf_{\kappa\to 0+} \kappa \log Q^\kappa[O]
\geq - \inf_{x \in O} \big( E(x) - \Phi(x)\big).
\end{align*}
\end{enumerate}
\end{lemA}

In the above formulation, the Radon-Nikodym derivative between the two measures is expressed in terms of the function $\Phi$ independent of $\kappa$, restricting its applicability to a very rigid class of changes of measures. 
To get a general principle, we may conversely express $\Phi = \Phi^\kappa$ in terms of the Radon-Nikodym derivative as $\Phi^\kappa = \kappa\log \big(\frac{\ud P^\kappa}{\ud Q^\kappa} \big)$, 
and then study what properties the functions $(\Phi^\kappa)_{\kappa > 0}$ need to satisfy in order for the conclusion of Varadhan's lemma to still hold. 
We formulate a sufficient property as ``steady convergence'' of the Radon-Nikodym derivatives,  
$\Phi^\kappa \xrightarrow{\kappa \to 0} \Phi^0$, in the sense of the following definition.

\begin{df}\label{def:steady-convergence}
Let $X$ be a topological space.  
A family of functions $f^\kappa \colon X \to [-\infty, +\infty]$, $\kappa > 0$, is said to converge \emph{steadily} to $f^0 \colon X \to [-\infty, +\infty]$ as $\kappa \to 0$ if the following holds.
\begin{enumerate}[label=(\roman*)]
\item\label{item:uniform-convergence} 
For every $M \in (0,\infty)$, 
the convergence $f^\kappa \xrightarrow{\kappa \to 0} f^0$ is uniform on $(f^0)^{-1}[-M, M]$: 
for every $\varepsilon > 0$, there exists $\kappa_\varepsilon = \kappa_\varepsilon(M) \in (0,\infty)$ such that
\begin{align*}
|f^0(x)| \le M\qquad \implies \qquad|f^\kappa(x)-f^0(x)| < \varepsilon, \qquad \textnormal{for all } \kappa \in (0, \kappa_\varepsilon).
\end{align*}

\item\label{item:tail-estimates} 
For every sufficiently large $M \in (0,\infty)$, there exist $\omega_M, \kappa_M  \in (0,\infty)$ such that
\begin{align*}
\begin{split}
f^0(x) & \ge  \hphantom{-} M \qquad\; \implies \qquad f^\kappa(x) \ge \hphantom{-} \omega_M, \phantom{-}\\
f^0(x) & \le -M \;\qquad \implies \qquad f^\kappa(x) \le -\omega_M,
\end{split} \qquad \textnormal{for all } \kappa \in (0, \kappa_M),
\end{align*}
and $\underset{M \to \infty}{\lim} \, \omega_M = +\infty$.
\end{enumerate}
\end{df}

Observe that if $f^0$ is bounded from above, the case $f^0(x) \ge M$ in~\ref{item:tail-estimates} is vacuous for $M$ sufficiently large, and is thus automatically satisfied.

\begin{thm}\label{thm:Varadhan-general}
Suppose that the probability measures $(P^\kappa)_{\kappa>0}$ satisfy an LDP in a topological space $X$
with good rate function $E$. 
For each $\kappa > 0$, let $\Phi^\kappa \colon X \to [-\infty, +\infty]$ be a measurable function, 
and let $Q^\kappa$ be a measure on $X$ absolutely continuous with respect to $P^\kappa$ with Radon-Nikodym derivative 
\begin{align*}
\frac{\ud Q^\kappa}{\ud P^\kappa}(x) = \exp \Big(\frac{1}{\kappa} \Phi^\kappa(x) \Big) , \qquad x \in X . 
\end{align*}
Suppose that as $\kappa \to 0$, the functions $\Phi^\kappa$ converge steadily to $\Phi^0 \colon X \to [-\infty, +\infty]$,
and that $\Phi^0$ is bounded from above. Then, the following hold.
\begin{enumerate}
\item \label{item1_Var-gen} 
If $\Phi^0$ is upper semicontinuous on a closed subset $F$ of $X$, then
\begin{align*}
\limsup_{\kappa\to 0+} \kappa \log Q^\kappa[F]
\leq - \inf_{x \in F} \big(E(x) - \Phi^0(x)\big) .
\end{align*}

\item \label{item2_Var-gen} 
If $\Phi^0$ is lower semicontinuous on an open subset $O$ of $X$, then
\begin{align*}
\liminf_{\kappa\to 0+} \kappa \log Q^\kappa[O]
\geq - \inf_{x \in O} \big( E(x) - \Phi^0(x)\big).
\end{align*}

\item \label{item3_Var-gen} 
If $\Phi^0$ is continuous, the limit 
\begin{align}\label{eq:total_mass_limit}
\underset{\kappa\to 0+}{\lim} \, \kappa \log Q^\kappa[X] = - \underset{x \in X}{\inf} \, (E(x)-\Phi^0(x)) \; \in \; [-\infty,+\infty] 
\end{align}
exists. 
If the limit does not equal $-\infty$, 
then the \textnormal{(}well-defined\textnormal{)} probability measures 
\begin{align*}
\hat Q^\kappa := \frac{Q^\kappa}{Q^\kappa[X]} , \qquad \textnormal{for small enough } \kappa>0 ,
\end{align*}
satisfy an LDP with good rate function $J \colon X \to [0, \infty]$ given by
\begin{align}\label{eq:rate_gen_J}
J(x) := E(x) - \Phi^0(x) - \inf_{y \in X} \big( E(y) - \Phi^0(y) \big) , \qquad x \in X .
\end{align}
Moreover, if $Q^\kappa$ are probability measures, then $J(x) = E(x) - \Phi^0(x)$. 
\end{enumerate}
\end{thm}

\begin{proof}
Let $M, \varepsilon > 0$, and let $\kappa_\varepsilon, \kappa_M, \omega_M > 0$ be as in \Cref{def:steady-convergence}~\ref{item:uniform-convergence}~\&~\ref{item:tail-estimates}; without loss of generality, let us assume that $\omega_M < M$. 

\begin{itemize}[leftmargin=*]
\item [\ref{item1_Var-gen}.]
As the set $F_M := (\Phi^0)^{-1}[-M, \infty]$ is closed by the upper semicontinuity of $\Phi^0$ on $F$, we have
\begin{align*}
\Phi^\kappa(x) \le\;& \begin{cases}
\Phi^0(x) + \varepsilon, & \textnormal{when } x \in F_M \textnormal{ and } \kappa < \kappa_\varepsilon, \qquad\;\, [\textnormal{by \Cref{def:steady-convergence}~\ref{item:uniform-convergence}}],\\
-\omega_M, & \textnormal{when } x \notin F_M \textnormal{ and } \kappa < \kappa_M ,\qquad [\textnormal{by \Cref{def:steady-convergence}~\ref{item:tail-estimates}}],\\
\end{cases}\\
=:\;& \Phi^0_{M, \varepsilon}(x).
\end{align*}
For $\varepsilon < M-\omega_M$, the function $x \mapsto \Phi^0_{M,\varepsilon}(x)$ is upper semicontinuous and bounded from above (as $\Phi^0$ is bounded from above by assumption), 
so Item~\ref{item1_Var} of \Cref{lemma:Varadhan} yields
\begin{align*}
\limsup_{\kappa\to 0+} \kappa \log Q^\kappa[F]
\leq \; & - \inf_{x \in F} \big(E(x) - \Phi^0_{M, \varepsilon}(x)\big) 
\\
=\;& -\min\Big\{\inf_{x \in F \cap  F_M}\big(E(x)-\Phi^0(x)+\varepsilon\big), \inf_{x \in F \smallsetminus F_M}\big(E(x)+\omega_M\big)\Big\}
\\
\underset{\varepsilon \to 0}{\overset{M \to \infty}{\longrightarrow}} \;& 
-\inf_{x \in F \cap F_\infty}(E(x)-\Phi^0(x)) = -\inf_{x \in F}(E(x)-\Phi^0(x)),
\end{align*}
where $F_\infty := \bigcup_{M>0}F_M$. 
(The last equality holds since $E(x)-\Phi^0(x) = +\infty$ on $X\smallsetminus F_\infty$.)

\item [\ref{item2_Var-gen}.]
Write $\tilde J = E - \Phi^0$. 
Without loss of generality, we assume $\tilde J(O) := \inf_{x \in O} J(x) = M < \infty$. 
Choose $x_\varepsilon \in O$ such that $\tilde J(x_\varepsilon) \le M+\varepsilon$, so that in particular, $\Phi^0(x_\varepsilon) > -\infty$. 
Since $\Phi^0$ is lower semicontinuous on $O$, $x_\varepsilon$ has an open neighborhood $O_\varepsilon \subset O$ such that $\Phi^0 \ge \Phi^0(x_\varepsilon)-\varepsilon$ on $O_\varepsilon$. 
Let $x \mapsto \Phi^0_\varepsilon(x)$ be the lower-semicontinuous function which equals $\Phi^0(x)$ on $O_\varepsilon$ and $\Phi^0(x_\varepsilon)-2\varepsilon$ on $X \smallsetminus O_\varepsilon$.
In particular, $\Phi^0_\varepsilon$ is also bounded from above (as $\Phi^0$ is bounded from above by assumption). 
Writing $M_\varepsilon := \sup_{x \in O_\varepsilon} |\Phi^0_\varepsilon(x)|$, 
we see that 
\begin{align*}
\Phi^\kappa(x) \ge \Phi^0_\varepsilon(x) - \varepsilon , \qquad x \in O_\varepsilon , \; \kappa < \kappa_\varepsilon(M_\varepsilon) ,
\end{align*}
where $\kappa_\varepsilon(M_\varepsilon)$ is chosen as in \Cref{def:steady-convergence}~\ref{item:uniform-convergence}. 
Thus, Item~\ref{item2_Var} of \Cref{lemma:Varadhan} yields
\begin{align*}
\liminf_{\kappa\to 0+} \kappa \log Q^\kappa[O]
\geq - \inf_{x \in O_\varepsilon} \big( E(x) - \Phi^0_\varepsilon(x) + \varepsilon\big)
\ge\;& -M-2\varepsilon \quad \xrightarrow{\varepsilon \to 0} \quad -M = - \tilde J(O).
\end{align*}

\item [\ref{item3_Var-gen}.]
The LDP bounds from Items~\ref{item1_Var-gen}--\ref{item2_Var-gen} proven above imply that the limit~\eqref{eq:total_mass_limit} exists in $[-\infty,+\infty]$. 
If the limit does not equal $-\infty$, then because 
\begin{align*}
0 < Q^\kappa[X] \leq \exp\bigg(\frac{1}{\kappa}\underset{x \in X}{\sup}\,\Phi^\kappa(x)\bigg) < \infty,
\end{align*}
we see that the probability measures $\hat Q^\kappa$ are well-defined for small enough $\kappa>0$.
If $Q^\kappa$ are probability measures, then obviously $\kappa \log Q^\kappa[X] = 0$ for all $\kappa$.
Also, the map $x \mapsto J(x)$ defined by~\eqref{eq:rate_gen_J} is a well-defined lower semicontinuous function,
as a sum of the lower semicontinuous $E$, the continuous $-\Phi^0$, and a constant.
Thus, for each $M \ge 0$, the sub-level set $J^{-1}[0, M] \subset X$ is closed and contained in the compact set 
\begin{align*}
E^{-1} \Big[ 0, M + \sup_{x \in X} \Phi^0(x) + \inf_{y \in X} \big( E(y)-\Phi^0(y) \big) \Big] .
\end{align*}
This shows that $J$ is a good rate function. 
\qedhere
\end{itemize}
\end{proof}

The next result states that steadily converging sequences with limits bounded from above are stable under addition. In light of~\Cref{thm:Varadhan-general} this is useful: it may be easier to decompose $\Phi^\kappa$ into steadily converging pieces instead of proving steady convergence directly. Dually from the change of measure perspective, the applicability of~\Cref{thm:Varadhan-general} is preserved when iterating changes of measures: 
if both $\frac{\ud Q^\kappa}{\ud P^\kappa} = \exp \big(\frac{\Phi^\kappa}{\kappa}\big)$ and $\frac{\ud\tilde Q^\kappa}{\ud Q^\kappa} = \exp \big(\frac{\tilde \Phi^\kappa}{\kappa}\big)$ 
satisfy the assumptions of~\Cref{thm:Varadhan-general}, then so does 
$\frac{\ud\tilde Q^\kappa}{\ud P^\kappa} = \exp\big(\frac{\Phi^\kappa+\tilde\Phi^\kappa}{\kappa}\big)$.

\begin{prop}\label{thm:steady-limit-addition}
Let $X$ be topological space and $(f^\kappa)_{\kappa > 0}$ and $(g^\kappa)_{\kappa > 0}$ sequences of maps from $X$ to $[-\infty, +\infty]$ which respectively converge steadily to functions $f^0, g^0 \colon X \to [-\infty, +\infty]$ bounded from above. 
Then, also $(f^\kappa + g^\kappa)_{\kappa > 0}$ converges steadily to $f^0 + g^0$.
\end{prop}

\begin{proof}
Since steady convergence is clearly stable under additions of constants, we may without loss of generality assume $f^0, g^0 \le 0$. Writing $h^\kappa := f^\kappa + g^\kappa$, we need to check the conditions~\ref{item:uniform-convergence} and~\ref{item:tail-estimates} from~\Cref{def:steady-convergence} for the function $(h^\kappa)_{\kappa \ge 0}$.
\begin{itemize}[leftmargin=*]
\item[\ref{item:uniform-convergence}] Since $h^0 \le \min(f^0, g^0)$, we have 
$(h^0)^{-1}[-M, 0] \subset ((f^0)^{-1}[-M, 0]) \cap ((g^0)^{-1}[-M, 0])$ for every $M \ge 0$, 
where $f^\kappa \xrightarrow{\kappa \to 0} f^0$ and $g^\kappa \xrightarrow{\kappa \to 0} g^0$ uniformly by steady convergence. 
Since uniform convergence is stable under addition, $h^\kappa \xrightarrow{\kappa \to 0} h^0$ uniformly on $(h^0)^{-1}[-M, 0]$. 

\item[\ref{item:tail-estimates}] 
Assuming $f^0, g^0$ are non-positive, $h^0(x) < -M$ implies $\min \{ f^0(x), g^0(x) \} < - \frac{M}{2}$. 
Hence, if for $M > 0$ sufficiently large we denote by $\omega^f_M, \omega^g_M > 0$ 
the constants from~\Cref{def:steady-convergence}~\ref{item:tail-estimates} for the maps $f^\kappa$ and $g^\kappa$ respectively, then we see that $h^0(x) < -M$ implies
\begin{align*}
h^\kappa(x) \le \min \big\{ f^\kappa(x), g^\kappa(x) \big\}
\le \max \big\{-\omega^f_{M/2}, -\omega^g_{M/2} \big\} =: -\omega_M \quad \xrightarrow{M \to \infty} \quad + \infty ,
\end{align*}
for sufficiently small $\kappa > 0$.  \qedhere
\end{itemize} 
\end{proof}

\bibliographystyle{alpha}

\end{document}

%% file: macro-arxiv.tex
\global\long\def\bR{\mathbb{R}}
\global\long\def\bZ{\mathbb{Z}}
\global\long\def\bN{\mathbb{N}}

\global\long\def\bC{\mathbb{C}}

\global\long\def\bD{\mathbb{D}}

\global\long\def\cX{\mathcal{X}}
\global\long\def\cL{\mathcal{L}}

\global\long\def\cT{\mathcal{T}}

\global\long\def\cL{\mathcal{L}}

\global\long\def\ud{\,\mathrm{d}}
\global\long\def\diam{\mathrm{diam}\,}
\global\long\def\dist{\mathrm{dist}\,}

\global\long\def\closed{F}

\global\long\def\np{N} 

\global\long\def\LogPartF{\mathcal{U}}
\global\long\def\PartF{\mathcal{Z}^{\kappa}}

\global\long\def\ii{\mathrm{i}}

\renewcommand{\liminf}{\varliminf}
\renewcommand{\limsup}{\varlimsup}

\newcommand{\SLE}{\operatorname{SLE}}

\global\long\def\one{\scalebox{1.1}{\textnormal{1}} \hspace*{-.75mm} \scalebox{0.7}{\raisebox{.3em}{\bf |}} }

\def\nBessel{J} 
\def\limnBessel{\smash{\cev{\nBessel}}}

\def\nDenergy{E} 
\def\nDindepenergy{\tilde E} 
\def\Denergy{E} 

\def\btheta{{\boldsymbol{\theta}}}
\def\bvartheta{\boldsymbol{\vartheta}}
\def\bsigma{{\boldsymbol{\sigma}}}

\def\bgamma{\boldsymbol{\gamma}}

\def\bgamma{{\boldsymbol{\gamma}}}

\def\charge{\mathfrak{c}}

\global\long\def\chamber{\cT_n}

\def\Pr{\mathsf{P}}
\def\Ex{\mathsf{E}}
\def\limPr{\smash{\cev{\Pr}}}

\def\prob{\mathbb P}

\def\probSLEindep{\prob^\kappa_{\bindeptime(T)}}

\def\bE{\mathbb E}

\global\long\def\blowuptime{\tau_{\mathrm{coll}}}

\global\long\def\BLoop{\mu^{\mathrm{loop}}}

\renewcommand{\vee}{\mathsf{v}}
\renewcommand{\u}{{\mathsf{u}}}

\newcommand{\rcap}{\textnormal{cap}}

\newcommand{\cev}[1]{\accentset{\leftarrow}{#1}}

\newcommand{\commonpaths}[1][]{\mathcal{C}_{#1}^n}
\newcommand{\limcommonpaths}{\smash{\cev{\mathcal{C}}}^n}

\newcommand{\Dindpaths}[1][]{\cX_{#1}^n}

\newcommand{\dcommonpaths}[1][]{d_{\commonpaths[#1]}}
\newcommand{\Dextpath}[2]{\cX_{\bindeptime(#1), #2}^n}
\newcommand{\DextpathCl}[2]{\smash{\overline{\cX}}^n_{\bindeptime(#1), #2}}

\newcommand{\dindpaths}[1][]{d_{\Dindpaths[#1]}}

\newcommand{\limembedding}{\smash{\cev\iota}}

\newcommand{\xescapeevent}[3]{\mathfrak{R}^{#1}_{#2,#3}}

\newcommand{\rhittingtime}[1][]{{\varrho}_{#1}}

\newcommand{\cond}{\;|\;}
\newcommand{\condbig}{\;\big|\;}

\newcommand{\indeptime}{\sigma}
\newcommand{\bindeptime}{\bsigma}

\newcommand{\A}{\mathbf a}
\newcommand{\B}{\mathbf b}
\newcommand{\M}{\mathbf M}

\newcommand{\RNloop}{\mathcal{M}}

\newcommand{\hmeas}[1]{{\operatorname{hm}_{#1}}}

\newcommand{\indepT}{{\tilde T}}
\newcommand{\bindepT}{\tilde{\boldsymbol{T}}}

\newcommand{\prDpaths}[1]{\pi_{#1}}
\newcommand{\prCommonTime}[1]{\uppi_{#1}}
\newcommand{\prCommonTimeInv}[1]{\uppi^{-1}_{#1}}

\newcommand{\mindex}{m}
\newcommand{\mindexMin}{m_0}
\newcommand{\Mindex}{M}

\newcommand{\cLrenorm}{\hat\cL}